\documentclass[letterpaper,12pt]{amsart}

\usepackage{amssymb}
\usepackage{amsmath}
\usepackage{algpseudocode, algorithm}

\numberwithin{equation}{section}
\numberwithin{algorithm}{section}

\newtheorem{lemma}{Lemma}[section]
\newtheorem{question}[lemma]{Question}
\newtheorem{theorem}[lemma]{Theorem}
\newtheorem{proposition}[lemma]{Proposition}
\newtheorem{definition}[lemma]{Definition}
\newtheorem{corollary}[lemma]{Corollary}

\newtheorem{conjecture}[lemma]{Conjecture}

\theoremstyle{definition}
\newtheorem{remark}[lemma]{Remark}
\newtheorem{example}[lemma]{Example}

\newcommand{\Z}{\mathbb{Z}}

\newcommand{\id}{\mathrm{id}}
\newcommand{\Sym}{\mathrm{Sym}}

\newcommand{\LRA}{\Leftrightarrow}

\newcommand{\inv}{\mathrm{inv}}

\newcommand{\Soc}{\mathrm{Soc}}

\newcommand{\Autgp}{\mathrm{Aut}_{\mathrm{gp}}}
\newcommand{\Autmag}{\mathrm{Aut}_\mathrm{mag}}
\newcommand{\Symo}{\mathrm{Sym}_0}
\newcommand{\ttau}{\tilde{\tau}}
\newcommand{\tsigma}{\tilde{\sigma}}
\newcommand{\tzeta}{\tilde{\zeta}}
\newcommand{\tpi}{\tilde{\pi}}
\newcommand{\bs}{\backslash}
\newcommand{\Fix}{\mathrm{Fix}}
\newcommand{\Supp}{\mathrm{Supp}}

\begin{document}
	\title{Magma Automorphisms and Quasi-Linear Cycle Sets}
	
	\author{Nigel P.~Byott}
	\address{Department of Mathematics \& Statistics, University of Exeter, Exeter 
		EX4 4QF, U.K.}  
	\email{N.P.Byott@exeter.ac.uk}
	
	\author{Edgar Jasko}
	\address{Department of Mathematics, KU Leuven, B-3001 Leuven, Belgium}  
\email{edgargladstone.jasko@student.kuleuven.be}	
	
	\thanks{This work was supported by the Engineering and Physical Sciences
		Research Council [grant number EP/V005995/1] and by the London Mathematical
		Society [Undergraduate Research Bursary URB-2024-37]  \newline 
		\indent
		For the purpose of open access, the authors have applied a CC BY public
		copyright licence  to any Author Accepted Manuscript version arising. \newline 
		\indent
		Data Access Statement: No data was used for the research described in this
		article.}

	\date{\today}
	
	\bibliographystyle{amsalpha}
	
\begin{abstract}

Rump asked in 2016 whether every finite quasi-linear cycle set is
retractible. This was recently proved by Dietzel. In this paper, we reinterpret Rump's question in terms of automorphisms of certain
pointed magmas, and propose a more general conjecture on the automorphisms of these magmas. We provide some theoretical and
computational evidence in support of our conjecture.

\end{abstract}
	
\maketitle

\section{Introduction}

The quantum Yang-Baxter equation first arose independently in the work of Yang \cite{Yang} on quantum systems and of Baxter \cite{Baxter} on statistical mechanics. It has since been found to play a fundamental role in mathematics and physics. In particular, it underlies the theory of quantum groups (see for example \cite{Majid}). A solution of the quantum Yang-Baxter equation on a vector space $V$ is a linear map $R: V \otimes V \to V \otimes V$ satisfying 
\begin{equation} \label{QYBE}
	 R_{12} R_{13} R_{23} = R_{23} R_{13} R_{12} : V \otimes V \otimes V \to V \otimes V \otimes V, 
\end{equation}
where $R_{ij}$ acts as $R$ in positions $i$ and $j$, and as the identity in the remaining position. Via composition with the twist map $x \otimes y \mapsto y \otimes x$, the quantum Yang-Baxter Equation is equivalent to the equation
\begin{equation} \label{YBE}
	   	R_{12} R_{23} R_{12} = R_{23} R_{12} R_{23},
\end{equation}
which is known simply as the Yang-Baxter Equation.

Drinfel'd \cite{Drinfeld} made the fruitful suggestion that combinatorial aspects of the quantum Yang-Baxter equation might be investigated by studying, for a nonempty set $X$, those functions $r : X \times X \to X \times X$ which satisfy the analogous condition to (\ref{QYBE}):
\begin{equation} \label{stQYBE}
	r_{12} r_{13} r_{23} = r_{23} r_{13} r_{12}.
\end{equation}
Any such set-theoretic solution of the quantum Yang-Baxter equation then gives rise to a solution of (\ref{QYBE}) on the vector space $V$ with $X$ as a basis. Similarly, if $r$ satisfies
\begin{equation} \label{stYBE}
	r_{12} r_{23} r_{12} = r_{23} r_{12} r_{23}
\end{equation}
then $r$ is a set-theoretic solution of the Yang-Baxter equation. 

We write a solution $r$ of (\ref{stQYBE}) in the form
$r(x,y) = (x^y, {}^x y)$. 
Then $r$ is said to be {\em left-nondegenerate} if the map $x \mapsto x^y$ is bijective for each $y$, and {\em nondegenerate} if also $y \mapsto {}^x y$ is bijective for each $x$. It is said to be {\em unitary} if $r_{21} r = \id_{X \times X}$. The corresponding solution of (\ref{stYBE}) is then {\em involutive}, i.e.~~$r^2=\id_{X \times X}$.

Set-theoretic solutions of the (quantum) Yang-Baxter equation have been intensively studied, beginning with groundbreaking work of Etingof, Schedler and Soloviev \cite{ESS} and Lu, Yan and Zhu \cite{LYZ}, and have been found to have connections with many other algebraic concepts, including quandles \cite{Fenn}, Garside groups \cite{Chouraqui} and skew lattices \cite{CV-V}. They have also led to the introduction of several new algebraic objects, such as braces \cite{Rump07}, skew braces \cite{GV} and cycle sets \cite{Rump05}. Braces give solutions of (\ref{stQYBE}) which are both nondegenerate and unitary. Conversely, given a nondegenerate unitary solution of (\ref{stQYBE}), one can form its structure group, which is canonically a brace. Skew braces generalise this picture to solutions which are nondegenerate but not necessarily unitary. 

In this paper, we will be concerned with cycle sets. 
There is a bijective correspondence between cycle sets and left-nondegenerate unitary solutions of (\ref{stQYBE}) \cite[Proposition 1]{Rump05}, or, equivalently, between cycle sets and solutions of (\ref{stYBE}) which are left-nondegenerate and involutive.
A cycle set consists of a set $X$ and a binary operation $\cdot$ on $X$ such that the function $y \mapsto x \cdot y$ is a permutation of $X$ for each $x \in X$ and the identity $( x \cdot y) \cdot (x \cdot z)  = (y \cdot x) \cdot (y \cdot z)$ holds for all $x$, $y$, $z \in X$.

In a subsequent paper \cite{Rump16}, Rump showed how examples of cycle sets arise from certain $\tau$-groups. For an abelian group $A$, we define
\begin{equation} \label{symo}
	 \Symo(A)=\{ \tau: \tau \mbox{ is a permutation of $A$ with } \tau(0)=0 \}. 
\end{equation} 
A $\tau$-group is an abelian group $A$ together with an element $\tau$ of $\Symo(A)$. We then introduce a binary operation $\cdot_\tau$ on $A$ by 
\begin{equation} \label{def-dot}
	a \cdot_\tau b = \tau(b-a) - \tau(-a). 
\end{equation} 	
The $\tau$-group $A$ is a cycle set if and only if 
\begin{equation}  \label{tau-autom}
	\tau(a \cdot b) = \tau(a) \cdot \tau(b) \mbox{ for all } a, b \in A. 
\end{equation}
A quasi-linear cycle set \cite[Definition 3]{Rump16} is a cycle set arising in this way from a $\tau$-group. Its socle is defined to be the subset
\begin{eqnarray*} 
	\Soc(A,\tau) & = & \{ a \in A : a \cdot b = 0 \cdot b \mbox{ for all } b \in A \} \\
	& = & \{ a \in A : \tau(b-a)=\tau(b) + \tau(-a) \mbox{ for all } b \in A \}.
\end{eqnarray*}
Then $\Soc(A,\tau)$ is a subgroup of $A$ \cite[Proposition 5]{Rump16}. Thus we may rewrite its definition as 
\begin{equation} \label{def-soc} 
  \Soc(A,\tau) =  \{ a \in A : \tau(b+a)=\tau(b) + \tau(a) \mbox{ for all } b \in A \}.
\end{equation} 
In particular, $\Soc(A,\tau)=A$ precisely when $\tau$ is a group automorphism of $A$. 
Moreover, by \cite[Proposition 12]{Rump16}, the quotient group $A/\Soc(A,\tau)$ becomes a quasi-linear cycle set. We call this the {\em retraction} of $A$, and we say that $A$ is {\em retractible} if $A$ can be reduced to a set of cardinality $1$ by iterating this operation finitely many times.

Rump \cite[Problem 2]{Rump16} posed the following question:

\begin{question} \label{Rump-qn}
	Does every finite quasi-linear cycle set of size $>1$ have non-trivial socle?
\end{question}

If Question \ref{Rump-qn} has an affirmative answer then every finite quasi-linear cycle set is
retractible. In particular, this would mean that if $(A,\tau) $ is any
quasi-linear cycle set whose cardinality is prime, then $\Soc(A,\tau)=A$ and $\tau$ must be a group automorphism of $A$. This prediction has been verified computationally by Rump for primes $p<17$. In unpublished work, I.~Colazzo and L.~Vendramin verified computationally that all quasi-linear cycle sets of size $<23$ have non-trivial socle. 

In this paper, we reinterpret Question \ref{Rump-qn} by viewing the $\tau$-group given by the abelian group $A$ and the permutation $\tau \in \Symo(A)$ as an object in the category of pointed magmas, that is, of sets with a binary operation and a distinguished element. The binary operation is that given by (\ref{def-dot}), and the distinguished element is the identity element $0$ of $A$. We will write $(A,\tau)$ for the pointed magma given by $A$ and $\tau$. 

It follows from (\ref{tau-autom}) that the pointed magma $(A,\tau)$ is a quasi-linear cycle set if and only if the permutation $\tau$ is itself a magma automorphism of $(A,\tau)$. This observation, together with an investigation of a number of examples, suggested the following conjecture:

\begin{conjecture}  \label{main-conj}
Let $A$ be a finite abelian group, and let $\tau$, $\sigma \in \Symo(A)$. If $\sigma$ is a magma automorphism of $(A,\tau)$ then either $\Soc(A,\tau)\neq \{0\}$ or $\Soc(A,\sigma)\neq \{0\}$.
\end{conjecture}

In the special case $\sigma=\tau$, Conjecture \ref{main-conj} predicts that if (\ref{tau-autom}) holds, so that $(A,\tau)$ is a quasi-linear cycle set, then $(A,\tau)$ has non-trivial socle. Thus Conjecture \ref{main-conj} implies an affirmative answer to Question \ref{Rump-qn}. We also note that the conclusion of Conjecture \ref{main-conj} is symmetric in $\sigma$ and $\tau$. In fact, a reciprocity holds between $\sigma$ and $\tau$ so that $\sigma$ is a magma automorphism of $(A,\tau)$ if and only if $\tau$ is a magma automorphism of $(A,\sigma)$ (see Proposition \ref{recip} below).

Our goal in this paper is to investigate Conjecture \ref{main-conj}. 

When this paper was essentially complete, we learned of the recent proof by Dietzel \cite{Dietzel} that every finite nontrivial quasilinear cycle set is retractable using the notion of cabling introduced by Lebed, Ramirez and Vendramin \cite{LRV}. This gives an affirmative answer to Question \ref{Rump-qn}. We believe that our results on the more general Conjecture \ref{main-conj} are still of interest. 

We will give a number of theoretical results consistent with Conjecture \ref{main-conj}, and verify the conjecture computationally for all abelian groups of order $\leq 14$ and also for the cyclic groups of order $17$ and $19$. We find it convenient to adopt the following terminology.

\begin{definition}  \label{def-c-ex}
Let $A$ be an abelian group and let $\tau \in \Symo(A)$. Then $\tau$ is a {\em counterexample for} $A$ if $\Soc(A,\tau)=\{0\}$ and there exists some $\sigma \in \Symo(A)$ such that  $\Soc(A,\sigma)=\{0\}$ and $\sigma$ is a magma automorphism of $(A,\tau)$.
\end{definition}

Thus Conjecture \ref{main-conj} holds for the abelian group $A$ if and only if there is no counterexample for $A$.

The second-named author's contribution to this project was funded by a London Mathematical Society Undergraduate Research Bursary (URB-2024-37) held for a 6-week period during the summer of 2024, and comprised both computational and theoretical work. This built on prior theoretical work of the first-named author in the case when $A$ has prime order. After the Bursary was completed, further work was undertaken by the first-named author to extend the results and computations beyond the prime order case. 

The nature of our computations is more combinatorial than group-theoretic. We therefore implemented them in Python, rather than using a group-theoretic package such as MAGMA or GAP. While it is probably possible to extend the computations using more sophisticated methods (e.g.~parallel processing, constrained optimisation), this is beyond the scope of the current project.

The paper is organised as follows. In \S2, we define the category of pointed magmas and give some general results on pointed magmas given by $\tau$-groups. In \S3, we consider the case where $\tau=(a,b)$ is a transposition, showing that if $A$ is cyclic and $\gcd(b-a,b)=1$ then Conjecture \ref{main-conj} holds for $(A,\tau)$. In \S4, we consider submagmas in our magmas $(A, \tau)$. 
In \S5, we give some results on the cycle structures of $\tau$ and $\sigma$ in various cases. Finally, \S6 describes our computational methods and results. Throughout the paper, we include a number of explicit examples.

\section{Pointed Magmas and $\tau$-groups}
A magma is a set together with a binary operation. We augment this notion by
including a distinguished element.

\subsection{The category of pointed magmas}
\begin{definition} \label{mag-cat}
\begin{itemize}
\item[(i)] 
A {\em pointed magma} $(X, \cdot,0)$ consists of a set $X$, a binary operation
$\cdot$ on $X$, and a distinguished element $0 \in X$.
\item[(ii)] Let $(X, \cdot_X,0_X)$ and $(Y, \cdot_, 0_Y)$ be pointed magmas. A {\em homomorphism of pointed magmas} from $X$ to $Y$ is a function $f:X \to Y$ such that $f(u \cdot_X v) = f(u) \cdot_Y f(v)$ for all $u$, $y \in X$, and $f(0_X)=0_Y$.
\item[(iii)] A {\em pointed submagma} of a pointed magma $(X,\cdot,0)$ is a subset $Y \subseteq X$ such that $0 \in Y$ and $u \cdot v \in Y$ for all $u$, $v \in Y$. 
\end{itemize}
\end{definition}
	
For the rest of this paper, all magmas considered will be pointed magmas. For brevity, we write
{\em magma homomorphism} (respectively, {\em submagma}) to mean {\em homomorphism of pointed magmas} (respectively, {\em pointed submagma}) in the sense of Definition \ref{mag-cat}. 

As usual, a bijective magma homomorphism will be called a {\em magma isomorphism}, and
a magma isomorphism $X \to X$ will be called a {\em magma automorphism} of $X$. Thus
we have the group $\Autmag(X)$ of magma automorphisms of the pointed magma $X$. A submagma $Y$ of $X$ is {\em proper} if $Y \neq X$, and {\em nontrivial} if $Y \neq \{0\}$.
		
\subsection{$\tau$-groups as pointed magmas}
Let $A$ be an abelian group. For $\tau \in \Symo(A)$, we have a binary operation $\cdot_\tau$ on $A$, as in (\ref{def-dot}):
$$ a \cdot_\tau b = \tau(b-a) - \tau(-a).  $$ 
This makes $A$ into a $\tau$-group, and $(A, \cdot_\tau,0)$ is a pointed magma. We denote this magma simply by $(A, \tau)$.

For each $a \in A$, the function $b \mapsto a \cdot b$ is a permutation of $A$. We denote this permutation by $\pi_a$. 

We write $\Soc(A,\tau)$ for the socle of the underlying $\tau$-group:
$$	\Soc(A,\tau) =  \{ a \in A : a \cdot_\tau b = 0 \cdot_\tau b \mbox{ for all } b \in A \}. $$
As noted in the introduction, $\Soc(A,\tau)$ is a subgroup of $A$, and $\Soc(A,\tau)=A$ precisely when $\tau$ is a group automorphism of $A$. 

We seek to investigate the group $\Autmag(A,\tau)$ of magma automorphisms of $(A,\tau)$. In particular, Conjecture \ref{main-conj} asserts that if $\sigma \in \Autmag(A,\tau)$ then either $\Soc(A,\tau) \neq \{0\}$ or $\Soc(A,\sigma) \neq \{0\}$. To avoid confusion, we write $\Autgp(A)$ for the group of  automorphisms of the abelian group $A$. 

\subsection{Examples}

We illustrate these ideas with some examples. Throughout the paper, we write $C_n$ for the cyclic group of order $n$ and identify this with $\Z/n\Z$. 

\begin{example} \label{C6-ex}
Let $A=C_6$ and $\tau=(135)(24)$. Table \ref{tab1} shows the operation $\cdot_\tau$ and the permutations $\pi_a$. 

\begin{table}[ht]
	\centerline{
		\begin{tabular}{|c|cccccc|c|} \hline
		$a$	& $0$ & $1$ & $2$ & $3$ & $4$ & $5$ &  $\pi_a$ \\ \hline
			$0$ & $0$ & $3$ & $4$ & $5$ & $2$ & $1$ & $(135)(24)$  \\ 
			$1$ & $0$ & $5$ & $2$ & $3$ & $4$ & $1$ & $(15)$ \\
			$2$ & $0$ & $5$ & $4$ & $1$ & $2$ & $3$ & $(153)(24)$ \\
			$3$ & $0$ & $3$ & $2$ & $1$ & $4$ & $5$ & $(13)$ \\
			$4$ & $0$ & $1$ & $4$ & $3$ & $2$ & $5$ & $(24)$ \\
			$5$ & $0$ & $1$ & $2$ & $5$ & $4$ & $3$ & $(35)$ \\ \hline
		\end{tabular}
		}  
	\vskip4mm
		\caption{Magma operation and permutations $\pi_a$ for $(C_6,\tau)$ with $\tau=(135)(24)$ (Example \ref{C6-ex}).  } 	\label{tab1}
\end{table}

As no other row agrees with the row for $a=0$, we see immediately that  $\Soc(A,\tau)=\{0\}$. Thus Conjecture \ref{main-conj} predicts that for any $\sigma \in \Autmag(A,\tau)$ we should have $\Soc(A,\sigma) \neq \{0\}$. 

We can determine $\Autmag(A,\tau)$ as follows. If $\sigma \in \Autmag(A,\tau)$ then certainly $\sigma(0)=0$. Since $\pi_a$ has order $6$ only for $a=0$ and $a=2$, we must have $\sigma(2)=2$. As $4$ is the only element apart from $2$ occurring in a $2$-cycle in $\pi_2$, we must also have $\sigma(4)=4$. Hence $\sigma(1) \in \{1,3,5\}$. Suppose that $\sigma(1)=3$. Then 
$$ \sigma(5)=\sigma(1 \cdot_\tau 1)=\sigma(1) \cdot_\tau \sigma(1) = 3 \cdot_\tau 3 = 1 $$
and
$$ \sigma(3)=\sigma(5 \cdot_\tau 5)=\sigma(5) \cdot_\tau \sigma(5) = 1 \cdot_\tau 1 = 5. $$  
Hence $\sigma=(135)$. One can verify that $\sigma \in \Autmag(A,\tau)$. We will see in Example \ref{C6-ex2} that $\Soc(A,\sigma) \neq \{0\}$ as expected. Since $\Autmag(C_6,\tau)$ is a group and it contains a unique permutation taking $1$ to $3$, we conclude that
$$ \Autmag(C_6,\tau) = \{ \id, (135), (153) \}. $$
\end{example}

\begin{example} \label{C6-ex2}
Again let $A=C_6$, and now consider $(C_6,\sigma)$ with $\sigma=(135)$ as in Example \ref{C6-ex}. The operation for this is shown in Table \ref{tab2}.

\begin{table}[ht]
	\centerline{
		\begin{tabular}{|c|cccccc|c|} \hline
		$a$	& $0$ & $1$ & $2$ & $3$ & $4$ & $5$ & $\pi_a$  \\ \hline
			$0$ & $0$ & $3$ & $2$ & $5$ & $4$ & $1$ & $(135)$  \\ 
			$1$ & $0$ & $5$ & $2$ & $1$ & $4$ & $3$ & $(153)$ \\
			$2$ & $0$ & $3$ & $2$ & $5$ & $4$ & $1$ & $(135)$  \\ 
			$3$ & $0$ & $5$ & $2$ & $1$ & $4$ & $3$ & $(153)$ \\
			$4$ & $0$ & $3$ & $2$ & $5$ & $4$ & $1$ & $(135)$  \\ 
			$5$ & $0$ & $5$ & $2$ & $1$ & $4$ & $3$ & $(153)$ \\ \hline
		\end{tabular}
	}  
	\vskip4mm
\caption{Magma operation and permutations $\pi_a$ for $(C_6,\sigma)$ with $\sigma=(135)$ (Example \ref{C6-ex2}).} \label{tab2}
\end{table}

We see immediately from Table \ref{tab2} that $\Soc(A,\sigma)$ is the subgroup $\{0,2,4\}$ of $A$. From Example \ref{C6-ex} and the reciprocity property mentioned in the introduction, we predict that $\tau=(135)(24) \in \Autmag(A,\sigma)$. By similar reasoning to Example \ref{C6-ex}, we find that $\Autmag(A,\sigma)$ is in fact the cyclic group of order $6$ generated by $\tau$. 
\end{example}

\begin{example} \label{no-automs}
Now let $A=C_7$ and $\tau=(25463)$. The operation is shown in Table \ref{tab3}.

\begin{table}[ht]
	\centerline{
		\begin{tabular}{|c|ccccccc|c|} \hline
		$a$	& $0$ & $1$ & $2$ & $3$ & $4$ & $5$ & $6$ & $\pi_a$  \\ \hline
			$0$ & $0$ & $1$ & $5$ & $2$ & $6$ & $4$ & $3$ & $(25463)$  \\ 
			$1$ & $0$ & $4$ & $5$ & $2$ & $6$ & $3$ & $1$ & $(146)(253)$ \\
			$2$ & $0$ & $6$ & $3$ & $4$ & $1$ & $5$ & $2$ & $(135)$  \\ 
			$3$ & $0$ & $5$ & $4$ & $1$ & $2$ & $6$ & $3$ & $(1563)(24)$ \\
			$4$ & $0$ & $4$ & $2$ & $1$ & $5$ & $6$ & $3$ & $(14563)$  \\ 
			$5$ & $0$ & $4$ & $1$ & $6$ & $5$ & $2$ & $3$ &  $(1452)(36)$ \\ 
			$6$ & $0$ & $4$ & $1$ & $5$ & $3$ & $2$ & $6$ &  $(14352)$ \\ 	\hline
		\end{tabular}
	}  
		\vskip4mm
	\caption{Magma operation and permutations $\pi_a$ for $(C_7,\tau)$ with $\tau=(25463)$ (Example \ref{no-automs}).} \label{tab3}
\end{table}

Let $\sigma \in \Autmag(A,\tau)$. Then $\sigma(1)=1$ since the only permutation $\pi_a$ consisting of two $3$-cycles in $\pi_1$. But then
\begin{eqnarray*}
	\sigma(4) & = & \sigma(1).\sigma(1)=1 \cdot 1 = 4, \\
	\sigma(6) & = & \sigma(1).\sigma(4)=1 \cdot 4 = 6, \\
	\sigma(5) & = & \sigma(4).\sigma(4)=4 \cdot 4 = 5, \\	
	\sigma(3) & = & \sigma(1).\sigma(5)=1 \cdot 5 = 3, \\	
	\sigma(2) & = & \sigma(1).\sigma(3)=1 \cdot 3 = 2. 
\end{eqnarray*}
Hence $\sigma = \id$, and in this case $\Autmag(A,\tau)=\{\id\}$.
\end{example}

\subsection{Some properties of $\Autmag(A,\tau)$}

For $\tau \in \Symo(A)$, we define $\ttau \in \Symo(A)$ by
\begin{equation} \label{ttau}
	\ttau(a) = -\tau(-a). 
\end{equation}

\begin{proposition} \label{tilde-comm}
If $\sigma \in \Autmag(A,\tau)$ then $\sigma$ commutes with $\tau$ and $\ttau$.
\end{proposition}
\begin{proof}
For all $a$, $b \in A$, we have 
$\sigma(a \cdot_\tau b) = \sigma(a) \cdot_\tau \sigma(b)$, so that 
\begin{equation} \label{mag-aut-eqn}
	 \sigma( \tau(b-a) - \tau(-a)) = \tau(\sigma(b)-\sigma(a)) - \tau( - \sigma(a)). 
\end{equation} 
Setting $a=0$ in (\ref{mag-aut-eqn}), we obtain $\sigma(\tau(b))=\tau(\sigma(b))$
for all $b$. Setting $b=a$  in (\ref{mag-aut-eqn}), we obtain $\sigma(0-\tau(-a))=0-\tau(-\sigma(a))$. Thus $\sigma(\ttau(a))=\ttau(\sigma(a))$ for all $a$.
\end{proof}

In the special case that $\tau$ is a group automorphism, we can determine $\Autmag(A,\tau)$ precisely.

\begin{proposition} \label{autom}
If $\tau \in \Autgp(A)$ then 
$$ \Autmag(A,\tau) = \{ \sigma \in \Symo(A) : \sigma \tau = \tau \sigma \}. $$
\end{proposition}
\begin{proof}
By Proposition \ref{tilde-comm} it suffices to prove that if  $\tau \in \Autgp(A)$ and $\sigma \tau = \tau \sigma$ then $\sigma \in \Autmag(A,\tau)$. Now 
\begin{eqnarray*}
 \sigma(a \cdot_\tau b) & = & \sigma(\tau(b-a)-\tau(-a)) \\
                    &  = & \sigma(\tau(b)-\tau(a)+\tau(a)) \\
                    & = &  \sigma(\tau(b)) \\
                    & = &	\tau(\sigma(b)), 
\end{eqnarray*}  
while
$$ \sigma(a) \cdot_\tau \sigma(b) = \tau(\sigma(b)-\sigma(a)) -
\tau(-\sigma(a)) = \tau(\sigma(b)). $$
Hence $\sigma \in \Autmag(A,\tau)$. 
\end{proof}
	
We next prove the reciprocity property mentioned in the introduction.
	
\begin{proposition} \label{recip}
For arbitrary $\tau$, $\sigma \in \Symo(A)$, 
$$ \sigma \in \Autmag(A,\tau) \LRA \tau \in \Autmag(A,\sigma). $$
\end{proposition}
\begin{proof}
By symmetry, it suffices to prove one implication. So suppose that $\sigma \in \Autmag(A,\tau)$. Again, (\ref{mag-aut-eqn}) holds for all $a$, $b \in A$. Writing $c=-a$ and $d=b-a$, we therefore have
\begin{equation} \label{mag-autom1}
	\sigma( \tau(d) - \tau(c)) = \tau(\sigma(d-c)-\sigma(-c)) - \tau( -
	\sigma(-c))	\mbox{ for all } c, d \in A. 
\end{equation}
But by Proposition \ref{tilde-comm}, 
$$ \tau(-\sigma(-c)) = -\ttau(\sigma(-c)) = \sigma(\ttau(-c)) = \sigma(-\tau(c)).$$
Thus 
$$ \sigma( \tau(d) - \tau(c)) - \sigma(-\tau(c)) = \tau(\sigma(d-c)-\sigma(-c)) 
\mbox{ for all } c, d \in A, $$
so that $\tau(c) \cdot_\sigma \tau(d) = \tau(c
\cdot_\sigma d)$. Hence $\tau \in \Autmag(A,\sigma)$.	
\end{proof}

\subsection{Conjugation by a group automorphism}
We now show that conjugating $\tau$ by a group automorphism does not change the magma
isomorphism class of $(A,\tau)$.

\begin{proposition}  \label{conjug}
Let $\tau \in \Symo(A)$, let $\alpha=\Autgp(A)$, and let $\omega=\alpha \tau \alpha^{-1}$. Then $\alpha: (A,\tau) \to (A,\omega)$ is an isomorphism of pointed magmas.
\end{proposition}
\begin{proof}
Let $a$, $b \in A$, and set $u=\alpha(a)$, $v=\alpha(b)$. Then 
\begin{eqnarray*} 
 \alpha(a \cdot_\tau b) & = & \alpha(\tau(b-a))-\alpha(\tau(-a)) \\
					  & = & \alpha \tau (\alpha^{-1}(v) -\alpha^{-1}(u) ) - \alpha \tau(-\alpha^{-1}(u)) \\
                      & = & \omega(v-u) - \omega(-u) \\
                      & = & u \cdot_\omega v.
\end{eqnarray*} 
\end{proof}

Recall from Definition \ref{def-c-ex} that $\tau \in \Symo(A)$ is a counterexample for $A$ if, for some $\sigma \in \Symo(A)$, we have 
\begin{equation} \label{st-conds}
	\sigma \in \Autmag(A,\tau), \qquad \Soc(A,\sigma) = \Soc(A,\tau) = \{0\}, 
\end{equation}

\begin{corollary} \label{conj-orbits}
If $\tau$ is a counterexample for A then so is $\alpha \tau \alpha^{-1}$ for every $\alpha \in \Autgp(A)$.	
\end{corollary}
\begin{proof}
Suppose that $\sigma$ is as in (\ref{st-conds}), and set $\omega = \alpha \tau \alpha^{-1}$ and $\rho = \alpha \sigma \alpha^{-1}$. By Proposition \ref{conjug}, $\rho \in \Autmag(A,\omega)$ and $\Soc(A,\omega)=\Soc(A,\rho)=\{0\}$. Hence $\omega$ is a counterexample for $A$.   
\end{proof}

\subsection{Some properties of the socle}

In general, $\Soc(A,\tau)$ need not be a submagma of $(A,\tau)$, and need not be fixed by $\tau$.

\begin{example}  \label{soc-not-submag}
	Let $A=C_3 \times C_3$. For brevity write the element $(a,b)\in A$ as $ab$. Let $\tau=(01,10)(02,20,12,21,22)$. The magma operation for $(A,\tau)$ is shown in Table \ref{C3C3-ex}. Here $\Soc(A,\tau)=\{00,01,02\}$, which is not a submagma since, for example $01 \cdot 01 =10 \not \in \Soc(A,\tau)$.  Moreover $\tau(\Soc(A,\tau))=\{00,10,20\} \neq \Soc(A,\tau)$.
\end{example}
\begin{table}[ht]
	\centerline{
		\begin{tabular}{|c|ccccccccc|} \hline
			& $00$ & $01$ & $02$ & $10$ & $11$ & $12$ & $20$ & $21$ & $22$ \\ \hline
			$00$ & $00$ & $10$ & $20$ & $01$ & $11$ & $21$ & $12$ & $22$ & $02$ \\ 
			$01$ & $00$ & $10$ & $20$ & $01$ & $11$ & $21$ & $12$ & $22$ & $02$ \\ 
			$02$ & $00$ & $10$ & $20$ & $01$ & $11$ & $21$ & $12$ & $22$ & $02$ \\ 			
			$10$ & $00$ & $10$ & $20$ & $21$ & $01$ & $11$ & $22$ & $02$ & $12$ \\ 
			$11$ & $00$ & $10$ & $20$ & $21$ & $01$ & $11$ & $22$ & $02$ & $12$ \\ 
			$12$ & $00$ & $10$ & $20$ & $21$ & $01$ & $11$ & $22$ & $02$ & $12$ \\ 
			$20$ & $00$ & $10$ & $20$ & $11$ & $21$ & $01$ & $02$ & $12$ & $22$ \\ 
			$21$ & $00$ & $10$ & $20$ & $11$ & $21$ & $01$ & $02$ & $12$ & $22$ \\ 
			$22$ & $00$ & $10$ & $20$ & $11$ & $21$ & $01$ & $02$ & $12$ & $22$ \\ \hline
		\end{tabular}
	}  
	\vskip4mm
	\caption{Magma operation for Example \ref{soc-not-submag}} 	\label{C3C3-ex}
\end{table}

\begin{proposition} \label{inv-tau}
	For any $\tau \in \Symo(A)$, we have
	$$\Soc(A,\tau^{-1}) = \tau^{-1}(\Soc(A,\tau)). $$
\end{proposition}
\begin{proof}
	For $a \in A$, we have
	\begin{eqnarray*}
		\lefteqn{\tau(a) \in \Soc(A,\tau^{-1}) } \\ & \LRA & \tau^{-1}(\tau(b)+\tau(a)) = \tau^{-1}(\tau(b)) + \tau^{-1}(\tau(a)) \mbox{ for all }b \in A \\
		& \LRA & \tau(b) + \tau(a) = \tau(b+a) \mbox{ for all }b \in A \\
		& \LRA & a \in\Soc(A,\tau).
	\end{eqnarray*}	
	The result follows.
\end{proof}

We can say more when $A$ is cyclic.

\begin{proposition} \label{Soc-k-prop}
	Let $A$ be a finite cyclic group, and let $\tau \in \Symo(A)$. For all $k \geq 1$, we have
	\begin{equation} \label{Soc-k} 
		\Soc(A,\tau) \subseteq \Soc(A,\tau^k) 
	\end{equation}
	with equality if $k$ is coprime to the order of $\tau$.
\end{proposition}
\begin{proof}
	It follows from (\ref{def-soc}) that restricting $\tau$ to $\Soc(A,\tau)$ gives an injective group homomorphism $\Soc(A,\tau) \to A$. Its image is therefore a subgroup of $A$ isomorphic to $\Soc(A,\tau)$. Since $A$ is cyclic, we have $\tau(\Soc(A,\tau))=\Soc(A,\tau)$. Inductively, for all $k \geq 1$ we have
	$$ \tau^k(\Soc(A,\tau)) = \Soc(A,\tau) $$
	and
	\begin{equation}  \label{tau-k}
		\tau^k(a+b)=\tau^k(a) + \tau^k(b) \mbox{ for } a \in \Soc(A,\tau), b \in A. 
	\end{equation}
	
Now (\ref{Soc-k}) follows immediately from (\ref{tau-k}). If $\tau$ has order $m$ and $\gcd(k,m)=1$, take $h \geq 1$ with $kh \equiv 1 \pmod{m}$. Then
$$ 	\Soc(A,\tau) \subseteq \Soc(A,\tau^k) \subseteq \Soc(A,(\tau^k)^h)=\Soc(A,\tau), $$  
giving equality in (\ref{Soc-k}).
\end{proof}

\begin{corollary} \label{cyclic-powers}
Let $A$ be a finite cyclic group and let $\tau$ be a counterexample for $A$. Let $m$ be the order of $\tau$. If $\gcd(k,m)=1$ then $\tau^k$ is also a counterexample for $A$.
\end{corollary}
\begin{proof}
Let $\sigma \in \Symo(A)$ be such that $\sigma$, $\tau$ satisfy (\ref{st-conds}). Then $\tau \in \Autmag(A,\sigma)$ by Proposition \ref{recip}, and, since $\Autmag(A,\sigma)$ is a group, $\tau^k \in \Autmag(A.\sigma)$. Using Proposition \ref{recip} again, we have $\sigma \in \Autmag(A,\tau^k)$, and, by Proposition \ref{Soc-k-prop}, $\Soc(A,\tau^k)=\{0\}$. Hence $\tau^k$ is a counterexample for $A$. 
\end{proof}

\subsection{Restatement of Conjecture \ref{main-conj} for $|A|$ prime}

\begin{proposition}
Let $A$ be a group of prime order, and let $\tau \in \Symo(A)$. Then Conjecture \ref{main-conj} for $(A,\tau)$ is equivalent to the following description of $\Autmag(A,\tau)$:
$$ \Autmag(A,\tau) = \begin{cases}
 \{\sigma \in \Symo(A) : \tau \sigma = \sigma \tau\} & \mbox{if } \tau \in \Autgp(A), \\
 \{\sigma \in \Autgp(A) : \tau \sigma = \sigma \tau\} & \mbox{if } \tau \not \in \Autgp(A).
\end{cases} $$
\end{proposition}
\begin{proof}
For arbitrary $A$, Proposition \ref{autom} shows unconditionally that if $\tau \in \Autgp(A)$ then $\Autmag(A,\tau) = \{\sigma \in \Symo(A) : \tau \sigma = \sigma \tau\}$.

Now suppose $A$ has prime order and that $\tau \not \in \Autgp(A)$. Then $\Soc(A,\tau) \neq A$, so $\Soc(A,\tau)=\{0\}$. If 
$\Autmag(A,\tau) = \{\sigma \in \Autgp(A) : \tau \sigma = \sigma \tau\}$ then, for any $\sigma \in \Autmag(A,\tau)$, we have  $\Soc(A,\sigma)=A \neq \{0\}$. Thus Conjecture \ref{main-conj} holds for $(A,\tau)$. Conversely, assume that Conjecture \ref{main-conj} holds for $(A,\tau)$. If $\sigma \in \Autmag(A,\tau)$ then $\Soc(A,\sigma) \neq \{0\}$, so $\Soc(A,\sigma)=A$ and $\sigma \in \Autgp(A)$. Moreover $\tau \sigma = \sigma \tau$ by Proposition \ref{tilde-comm}. This shows the inclusion $\Autmag(A,\tau) \subseteq  \{\sigma \in \Autgp(A) : \tau \sigma = \sigma \tau\}$. For the reverse inclusion, let $\sigma \in \Autgp(A)$ with $\tau \sigma =\sigma \tau$. Then $\tau \in \Autmag(A,\sigma)$ by the first case of the Proposition (with $\tau$ and $\sigma$ interchanged), so $\sigma \in \Autmag(A,\tau)$ by Proposition \ref{recip}. 
\end{proof}
 	
\section{Transpositions}

In this section, we prove that Conjecture \ref{main-conj} holds if $A$ is cyclic of prime order and $\tau$ is a transposition. In fact we prove a slightly more general result in Theorem \ref{trans-thm}. 

For an abelian group $A$ and a permutation $\tau \in \Symo(A)$, we define the support of $\tau$ to be
$$ \Supp(\tau) = \{a \in A : \tau(a) \neq a \}. $$

\begin{proposition} \label{S-lem}
Let $A$ be an arbitrary abelian group and let $\tau \in \Symo(A)$. Let
$$ S= \Supp(\ttau) = \{ a \in A : -\tau(-a) \neq a \},  $$
and let $\sigma \in \Autmag(A,\tau)$. Then, for each $b \in A$, we have  
	$$ \sigma(b+S) \backslash S =(\sigma(b)+S) \backslash S. $$
\end{proposition}
\begin{proof}
First observe that $\sigma(S)=S$ since, by Proposition \ref{tilde-comm}, $\sigma$ commutes with $\ttau$. 

For $a \in A \bs S$ and $b \in A$, we have 
\begin{eqnarray*}
  a \in b+S & \LRA & a-b \in S \\
  			& \LRA & \tau(b-a) \neq b-a \\
            & \LRA & \tau(b-a) - \tau(-a) \neq b-a -(-a) \\
            & \LRA & a \cdot_\tau b \neq b.
\end{eqnarray*}
Fix $b \in A$. For $a \in A$ we have $\sigma(a) \not \in S$ if and only if $a \not \in S$. Thus if $\sigma(a) \not \in S$, we have 
\begin{eqnarray*}
	\sigma(a) \in \sigma(b+S) & \LRA & a \in b+S \\
	                          & \LRA & a \cdot_\tau b \neq b \\
	                          & \LRA & \sigma(a) \cdot_\tau \sigma(b) \neq \sigma(b) \\
	                          & \LRA & \sigma(a)  \in \sigma(b) + S.
\end{eqnarray*}
Hence $\sigma(b+S) \bs S = (\sigma(b)+S) \bs S$.
\end{proof}

\begin{theorem} \label{trans-thm}
Let $A=C_n$ with $n>2$, and let $\tau \in \Symo(A)$ be a transposition $\tau=(ab)$ where $\gcd(b-a,n)=1$. Then	
\begin{equation} \label{transp-aut}
		\Autmag(A,\tau) = \begin{cases} \{\id\} & \mbox{if } b \neq -a; \\
		\{\id, \inv\} & \mbox{if } b =-a;
	\end{cases} 
\end{equation}
where $\inv(x)=-x$. In particular, Conjecture \ref{main-conj} holds for $(A, \tau)$. 
\end{theorem}
\begin{proof}

We begin with a preliminary observation. If $b=2a$ and $a=2b$ in $A$ then $3(b-a)=0$ and the coprimality hypothesis forces $n=3$. For $n=3$, we necessarily have $\tau=(12)=\inv$ and $\Autmag(A,\tau)=\{\id, \tau \}=\Autgp(A,\tau)$. Thus the Theorem holds for $n=3$. We henceforth assume that $n>3$, and further (by swapping $a$ and $b$ if necessary) that $b \neq 2a$. 

In the notation of Proposition \ref{S-lem}, we have 
$$ S= \Supp(\ttau)=	\{-a,-b\}. $$ 
For $0 \leq j \leq n-1$, let $T_j=a+j(a-b)+S = \{j(a-b), (j+1)(a-b)\}$. In particular, 
$$  	T_0 =  \{ 0, a-b\},  \qquad 	T_{n-1} = \{ b-a, 0 \}.  $$ 
The sets $T_j$ are pairwise distinct by the coprimality hypothesis, and for each $c \neq 0$ in $A$ there is a unique index $i$ such that 
$$ c \in T_j \LRA j=i \mbox{ or } i+1. $$ 
Moreover, there is a unique $j_0$ with $1 \leq j_0 \leq n-2$ such that $T_{j_0} = S= \{-a,-b\}$. We then have $T_{j_0-1} = \{b-2a, -a\}$ and $T_{j_0+1} = \{ -b, a-2b \}$. 

Next consider the sets $T'_j=T_j \bs S$. We have $T'_{j_0}= \emptyset$ and 
$$ T'_{j_0-1}=\{b-2a\}=\{(j_0-1)(a-b)\}, \quad T_{j_0+1} = \{(j_0+2)(a-b)\} = \{a-2b \}, $$ 
while $T'_j=T_j$ if $j \neq j_0$, $j_0 \pm 1$. 

Now let $\sigma \in \Autmag(A,\tau)$. By Proposition \ref{S-lem}, $\sigma$ permutes the sets $T'_j$. In particular, as $\sigma(0)=0$, the two sets $T'_0$, $T'_{n-1}$ containing $0$ must either be fixed by $\sigma$ or swapped by $\sigma$. 
We consider these two possibilities separately.
	
First suppose that $\sigma(T'_0)=T'_0$ and $\sigma(T'_{n-1}) = T'_{n-1}$. We will show that in this case $\sigma=\id$.
We first show by induction that for $0 \leq k \leq j_0-1$ we have 
\begin{equation} \label{transp-ind}
	\sigma(T'_k)=T'_k, \qquad \sigma(k(a-b))=k(a-b).
\end{equation}
This is true for $k=0$. Suppose that (\ref{transp-ind}) holds for some $k<j_0-1$. Then $T'_k=\{k(a-b),(k+1)(a-b)\}$ is fixed by $\sigma$. Since $\sigma$ fixes $k(a-b)$, it must also fix $(k+1)(a-b)$. Since $(k+1)(a-b) \in T'_j$ if and only if $j=k$ or $k+1$, and $\sigma$ fixes $T'_k$, it must also fix $T'_{k+1}$. Hence (\ref{transp-ind}) also holds for $k+1$, completing the induction. A similar argument, using descending induction and starting from $\sigma(T'_{n-1})=T'_{n-1}$, shows that (\ref{transp-ind}) also holds for $n-1 \geq k \geq j_0+2$. Thus $\sigma(c)=c$ for all $c \in A$ with the possible exceptions of $c=j_0(b-a)=-a$ and $c=(j_0+1)(b-a)=-b$. Hence $\sigma=\id$ or $\sigma=(-a,-b)=\ttau$. We will show that $\ttau$ cannot be a magma automorphism when $n>3$. Let $x=-2a$ and $y=b-2a$. Since $a \neq b$ and we have assumed that $b \neq 2a$, we have $\tau(-x)=-x$ and $\ttau(x)=x$. As $n>2$ and $\gcd(b-a,n)=1$, we have $2(b-a) \neq 0$, so that $b-2a \neq -a$, $-b$. Thus $\ttau(y)=y$. We now calculate
\begin{eqnarray*}
	 \ttau( x \cdot_\tau y) & = & \ttau(\tau(y-x)-\tau(-x)) \\
	                         & = & \ttau(\tau(b)-\tau(-x)) \\
	                         & = & \ttau(a-2a) \\
	                         & = & -b
\end{eqnarray*}
and
\begin{eqnarray*}
	\ttau(x) \cdot_\tau \ttau(y) & = & \tau( \ttau(y)-\ttau(x)) - \tau( -\ttau(x)) \\ 
								& = & \tau(y-x) - \tau(-x) \\
								& = & \tau(b)-\tau(-x) \\
								& = & a -(-x) \\
								& = & -a.						
\end{eqnarray*}        
As $-b \neq -a$ in $A$, this shows that $\ttau \not \in \Autmag(A,\tau)$.

Now suppose that $T'_0$ and $T'_{n-1}$ are swapped by the magma automorphism $\sigma$. We will show that in this case the only possibility is $\sigma =\inv$, and this can only occur if $b=-a$. 
We first show by induction that for $0 \leq k \leq j_0-1$, the following hold:
\begin{equation} \label{inv-bound}
	n-1-k \geq j_0+1; 
\end{equation}
\begin{equation} \label{inv-ind}
	\sigma(k(a-b))=k(b-a), \qquad  	\sigma(k(b-a))=k(a-b).
\end{equation}
\begin{equation} \label{inv-set}
	\sigma(T'_k) = T'_{n-1-k}, \qquad \sigma(T'_{n-1-k})=T'_k; 
\end{equation}
Clearly, these all hold when $k=0$. Suppose that they hold for some $k < j_0-1$. Then $|T'_k|=2$ and $T'_k=\{k(a-b), (k+1)(a-b)\}$. By (\ref{inv-set}), $\sigma$ swaps $T'_k$ and $T'_{n-1-k}$. Hence $|T'_{n-1-k}|=2$ as well, and $T'_{n-1-k}=\{k(b-a),(k+1)(b-a)\}$. In particular, $n-1-k \neq j_0+1$. Together with (\ref{inv-bound}), this shows that $n-1-k \geq j_0+2$, so that (\ref{inv-bound}) holds for $k+1$. Moreover, it follows from (\ref{inv-ind}) that $\sigma$ swaps $(k+1)(a-b)$ and $(k+1)(b-a)$, so (\ref{inv-ind}) holds for $k+1$. Finally, since 
$$  (k+1)(a-b) \in T'_j \LRA j=k \mbox{ or } j=k+1 $$
and
$$  (k+1)(b-a) \in T'_j \LRA j=n-1-k \mbox{ or } j= n-2-k, $$
it follows that $\sigma$ swaps $T'_{k+1}$ and $T'_{n-2-k}$. This shows that (\ref{inv-set}) holds for $k+1$, completing the induction. Now taking $k=j_0-1$ in (\ref{inv-bound}) and (\ref{inv-set}), we have that $n-j_0 \geq j_0+1$ and that $\sigma$ swaps $T'_{j_0-1}$ and $T'_{n-j_0}$. Thus $|T'_{n-j_0}|=|T'_{j_0-1}|=1$. Hence $n-j_0=j_0+1$, so that $n$ is odd and $j_0=\frac{1}{2}(n-1)$. As $T_{j_0}=\{j_0(a-b), (j_0+1)(a-b)\}=\{-a,-b\}$, we have $b=-a$. Since (\ref{inv-ind}) holds for $0 \leq k \leq \frac{1}{2}(n-1)$, we now know that $\sigma(c)=-c$ for all $c \in A$ with the possible exceptions of $c=\frac{1}{2}(n-1)(a-b)=-a=b$ and $c=\frac{1}{2}(n+1)(a-b)=-b=a$. Thus either $\sigma=\inv$ or $\sigma=\ttau \circ \inv$, where in this case $\ttau=\tau$. Now $\inv \in \Autgp(A)$ and $\inv$ commutes with $\tau$, so $\inv \in \Autmag(A,\tau)$ by Proposition \ref{autom}. We have already seen in the first case that $\ttau \not \in \Autmag(A,\tau)$. Thus $\ttau \circ \inv \not \in \Autmag(A,\tau)$. 

We have now shown both cases of (\ref{transp-aut}). In particular, $\Autmag(A,\tau) \subseteq \Autgp(A)$, so Conjecture \ref{main-conj} holds for $(A, \tau)$. 
\end{proof}

\section{Submagmas}

Recall from Definition \ref{mag-cat}(iii) that submagmas always contain $0$.

We begin with some easy examples of submagmas.

\begin{lemma}  \label{submags}
	Let $A$ be a finite abelian group, and let $\tau \in \Symo(A)$.
	\begin{itemize}
		\item[(i)] Any submagma of $(A,\tau)$ is a union of orbits of the subgroup $\langle \tau, \ttau\rangle$ of $\Symo(A)$. In particular, if $\langle \tau, \ttau\rangle$ is transitive on $A\bs\{0\}$ then the only submagmas are $\{0\}$ and $A$.
		\item[(ii)] The submagmas of size $2$ are the sets $\{0,a\}$ where $a$ satisfies $\tau(a)=\ttau(a)=a$.
		\item[(iii)] If $B$ is a subgroup of $A$ with $\tau(B)=B$ then $B$ is a submagma of $(A,\tau)$. 
		\item[(iv)] If $\sigma \in \Autmag(A,\tau)$ then the set of fixed points of $\sigma$
		$$   \Fix(\sigma)=\{ a \in A : \sigma(a) =a \} $$
		is a submagma of $(A,\tau)$. 
		\item[(v)] If $\tau \in \Autgp(A)$ and $B$ is any union of $\tau$-orbits containing $0$, then $B$ is a submagma of $(A,\tau)$. 
	\end{itemize}
\end{lemma}
\begin{proof}
	(i) If $B$ is a submagma of $(A,\tau)$ and $a \in B$ then $B$ contains $0 \cdot a = \tau(a)$ and $a \cdot a = \ttau(a)$.
	
	(ii) If $B=\{0,a\}$ is a submagma of size $2$ then $\tau(a)=0 \cdot a$ and $\ttau(a) = a \cdot a$ are in $A \bs\{0\}$, so $\tau(a)=\ttau(a)=a$. Conversely, if $\tau(a)=\ttau(a)=a$ then
	$$ 0 \cdot 0 = a \cdot 0 = 0, \quad 0 \cdot a = a \cdot a =a, $$
	so $\{0,a\}$ is a submagma.
	
	(iii) This is clear from the definition of $\cdot_\tau$.
	
	(iv) $0 \in \Fix(\sigma)$, and if $a$, $b \in \Fix(\sigma)$ then 
	$$ a \cdot_\tau b = \sigma(a) \cdot_\tau \sigma(b) = \sigma(a \cdot_\tau b), $$
	so $a \cdot_\tau b \in \Fix(\sigma)$. 
	
	(v) If $\tau \in \Autgp(A)$ and $B$ is a union of $\tau$-orbits then for $a$, $b \in B$ we have 
	$$ a \cdot_\tau b = \tau(b-a)-\tau(-a) = \tau(b)-\tau(a)+\tau(a)=\tau(b) \in B. $$
\end{proof}

\begin{remark} \label{no-Lagrange}
There is no obvious analogue of Lagrange's Theorem for pointed magmas. For example, the magma of cardinality $7$ in Example \ref{C6-ex2} has a submagma $\{0,1,2,3,5\}$ of cardinality $5$.
\end{remark}

\begin{definition}
	Let $S \subseteq A$. The {\em submagma of $(A,\tau)$ generated by $S$} is the smallest submagma containing $S$. We say $(A,\tau)$ is {\em cyclic} if it generated by a single element.
\end{definition}

If $(A,\tau)$ has no proper nontrivial submagmas then it is necessarily cyclic: any $a \neq 0$ generates the full magma $A$. 

\begin{remark}
When $A$ has prime order, proper submagmas of size greater than $2$ seem to be rather rare. For $|A|=5$ (respectively $7$, $11$), a computer search shows that $(A,\tau)$ contains a proper submagma of size at least $3$ in only for 8\% (respectively 3\%, 0.001\%) of permutations $\tau \in \Symo(A)$. This includes the cases where $\tau \in \Autgp(A)$, and, in particular, $\tau=\id_A$. 
\end{remark}

A cyclic magma may contain nontrivial proper submagmas, as shown by the next example.

\begin{example} \label{7-submag}
	Let $p=7$ and $\tau=(1,5)$. The operation is shown in Table \ref{has-submagma}. Then $(A,\tau)$ has a submagma $\{0,1,5\}$ of size 3. Since $\tau(3)=3=\ttau(3)$ and $\tau(4)=4=\ttau(4)$, it also has submagmas $\{0,3\}$ and $\{0,4\}$. The magma $(A,\tau)$ is however cyclic, generated by $2$.
\end{example}

\begin{table}[ht]
	\centerline{
		\begin{tabular}{|c|ccccccc|} \hline
			& $0$ & $1$ & $2$ & $3$ & $4$ & $5$ & $6$  \\ \hline
			$0$ & $0$ & $5$ & $2$ & $3$ & $4$ & $1$ & $6$   \\ 
			$1$ & $0$ & $1$ & $6$ & $3$ & $4$ & $5$ & $2$  \\
			$2$ & $0$ & $5$ & $6$ & $4$ & $1$ & $2$ & $3$   \\ 
			$3$ & $0$ & $4$ & $2$ & $3$ & $1$ & $5$ & $6$  \\
			$4$ & $0$ & $1$ & $5$ & $3$ & $4$ & $2$ & $6$   \\ 
			$5$ & $0$ & $1$ & $2$ & $6$ & $4$ & $5$ & $3$  \\ 
			$6$ & $0$ & $4$ & $5$ & $6$ & $3$ & $1$ & $2$  \\ 	\hline
		\end{tabular}
	}  
	\vskip4mm
	\caption{Magma operation for Example \ref{7-submag}} \label{has-submagma}
\end{table}	

If $|A|=p$ is prime, and $p<17$, then a computer search shows that every submagma $(A,\tau)$ with trivial socle is cyclic. This fails for $p=17$.

\begin{example} \label{non-cyclic-mag}
Let $A=C_{17}$ and let 
$$  \tau=(1,4)(2,8)(3,14)(5,12)(6,11)(7,10)(9,15)(13,16).  $$
The magma operation is shown in Table \ref{17-noncyc}. The magma is not cyclic, and in fact every $a \neq 0$ generates a submagma of size $3$. 
\end{example}

\begin{table}[ht]
	\centerline{
		\begin{tabular}{|c|ccccccccccccccccc|} \hline
			& $0$ & $1$ & $2$ & $3$ & $4$ & $5$ & $6$ & $7$ & $8$ & $9$ & $10$ & $11$ & $12$ & $13$ & $14$ & $15$ & $16$  \\ \hline
 $0$ &$0$&$4$&$8$&$14$&$1$&$12$&$11$&$10$&$2$&$15$&$7$&$6$&$5$&$16$&$3$&$9$&$13$ \\
$1$ &$0$&$4$&$8$&$12$&$1$&$5$&$16$&$15$&$14$&$6$&$2$&$11$&$10$&$9$&$3$&$7$&$13$ \\
$2$ &$0$&$4$&$8$&$12$&$16$&$5$&$9$&$3$&$2$&$1$&$10$&$6$&$15$&$14$&$13$&$7$&$11$ \\
$3$ &$0$&$6$&$10$&$14$&$1$&$5$&$11$&$15$&$9$&$8$&$7$&$16$&$12$&$4$&$3$&$2$&$13$ \\
$4$ &$0$&$4$&$10$&$14$&$1$&$5$&$9$&$15$&$2$&$13$&$12$&$11$&$3$&$16$&$8$&$7$&$6$ \\
$5$ &$0$&$11$&$15$&$4$&$8$&$12$&$16$&$3$&$9$&$13$&$7$&$6$&$5$&$14$&$10$&$2$&$1$ \\
$6$ &$0$&$16$&$10$&$14$&$3$&$7$&$11$&$15$&$2$&$8$&$12$&$6$&$5$&$4$&$13$&$9$&$1$ \\
$7$ &$0$&$16$&$15$&$9$&$13$&$2$&$6$&$10$&$14$&$1$&$7$&$11$&$5$&$4$&$3$&$12$&$8$ \\
$8$ &$0$&$9$&$8$&$7$&$1$&$5$&$11$&$15$&$2$&$6$&$10$&$16$&$3$&$14$&$13$&$12$&$4$ \\
$9$ &$0$&$13$&$5$&$4$&$3$&$14$&$1$&$7$&$11$&$15$&$2$&$6$&$12$&$16$&$10$&$9$&$8$ \\
$10$ &$0$&$9$&$5$&$14$&$13$&$12$&$6$&$10$&$16$&$3$&$7$&$11$&$15$&$4$&$8$&$2$&$1$ \\
$11$ &$0$&$16$&$8$&$4$&$13$&$12$&$11$&$5$&$9$&$15$&$2$&$6$&$10$&$14$&$3$&$7$&$1$ \\
$12$ &$0$&$16$&$15$&$7$&$3$&$12$&$11$&$10$&$4$&$8$&$14$&$1$&$5$&$9$&$13$&$2$&$6$ \\
$13$ &$0$&$11$&$10$&$9$&$1$&$14$&$6$&$5$&$4$&$15$&$2$&$8$&$12$&$16$&$3$&$7$&$13$ \\
$14$ &$0$&$4$&$15$&$14$&$13$&$5$&$1$&$10$&$9$&$8$&$2$&$6$&$12$&$16$&$3$&$7$&$11$ \\
$15$ &$0$&$6$&$10$&$4$&$3$&$2$&$11$&$7$&$16$&$15$&$14$&$8$&$12$&$1$&$5$&$9$&$13$ \\
$16$ &$0$&$4$&$10$&$14$&$8$&$7$&$6$&$5$&$11$&$3$&$2$&$1$&$12$&$16$&$5$&$9$&$13$ 	\\ \hline
\end{tabular}
}  
\vskip4mm
\caption{Magma operation for Example \ref{non-cyclic-mag}} \label{17-noncyc}
\end{table}	


\section{Cycle Structure}
In this section, we give several results on the decomposition of $\tau$ and $\sigma$ into disjoint cycles when $\sigma \in \Autmag(A,\tau)$. 

We will make frequent use of the following well-known result (see for example \cite[Proposition 1.1.1]{Sagan}.

\begin{proposition} \label{perm-comm}
Let $\pi$ be an element of the symmetric group $\Sym(n)$ on $n$ elements. For $1 \leq i \leq n$, let $m_i$ be the number of cycles of length $i$ in the expression for $\pi$ as a product of disjoint cycles. Then the centraliser of $\pi$ in $\Sym(n)$ is $Z_1 \times \cdots Z_n$, where $Z_i$ is generated by the $m_i$ cycles of length $i$ and the symmetric group $\Sym(m_i)$ permuting these cycles. Moreover $|Z_i| = i^{m_i} m_i!$ and $Z_i$ is isomorphic to the wreath product $C_i \wr \Sym(m_i)$. 
\end{proposition}

\subsection{$\tau$ of order divisible by a large prime}

\begin{lemma} \label{long-cycle}
Let $A$ be a finite abelian group of odd order $n \geq 5$, and let $\tau$, $\sigma \in \Symo(A)$ satisfy (\ref{st-conds}). Suppose that the order of $\tau$ is divisible by a prime $q  \geq (n+1)/2$.  Then $\sigma$ is supported on a set of cardinality at most $n-2-q$.
\end{lemma}
\begin{proof}
Clearly $q$ is odd and $\tau$ contains exactly one cycle $\zeta$ of length $q$. Let $X=A \backslash (\Supp(\zeta) \cup \{0\})$, a set of cardinality $n-1-q$. Then $\tau = \zeta \pi$, where $\pi$ is a permutation with $\Supp(\pi) \subseteq X$. Now $\sigma$ commutes with $\tau$, and it follows from Proposition \ref{perm-comm} that $\sigma = \zeta^a \rho$ for some $a$ with $0 \leq a \leq q-1$ and some permutation $\rho$ with $\Supp(\rho) \subseteq X$. Thus $\rho$ cannot contain a cycle of length $q$. But $\sigma$ also commutes with $\ttau=\tzeta \tpi$, and $\Supp(\tzeta) \neq \Supp(\zeta)$ as $q$ is odd. Hence $a=0$, and $\rho$ must fix $\Supp(\zeta) \cup \Supp(\tzeta)$, a set of cardinality at least $q+1$. It is therefore supported on a set of cardinality at most $n-2-q$.
\end{proof}

\subsection{The cyclic prime-power case}

Throughout this section, let $A=C_{p^d}$ be a cyclic group of order $p^d$, with $p$ prime and $d \geq 1$. In this case, $A$ has a unique minimal (nontrivial) subgroup $M$, namely the subgroup of order $p$. Thus if $\Soc(A,\tau) \neq \{0\}$ then $M \subseteq \Soc(A,\tau)$. We therefore have
\begin{equation}  \label{M-Soc}
 \Soc(A,\tau) \neq \{0\} \LRA \tau(a+b)=\tau(a)+\tau(b) \mbox{ for all } a \in M, b \in A. 
\end{equation}
This allows us to improve on Proposition \ref{Soc-k-prop}.

\begin{proposition}  \label{compose-soc}
If $\rho$, $\tau \in \Symo(A)$ with $\Soc(A,\rho) \neq \{0\} \neq \Soc(A,\tau)$ then $\Soc(A,\rho \tau) \neq \{0\}$.
\end{proposition}
\begin{proof}
For any $\sigma \in \Symo(A)$ we have 
$$ \Soc(A,\sigma) \neq \{0\} \LRA M \subseteq \Soc(A,\sigma). $$
In particular, $\rho$ and $\tau$ satisfy these equivalent conditions. Thus, using (\ref{M-Soc}), for $a \in M$ and $b \in A$ we have  
$$ \rho(\tau(a+b))=\rho(\tau(a)+\tau(b))=\rho(\tau(a)) + \rho(\tau(b)), $$
showing that $a \in \Soc(A,\rho \tau)$. Hence $M \subseteq \Soc(A,\rho \tau)$. 
\end{proof}
	
\begin{definition}
We call $\sigma \in \Symo(A)$ {\em special} if it satisfies the following properties:
\begin{itemize}
	\item[(i)] $\sigma$ has order $q^e$ where $q$ is prime and $e \geq 2$; 
	\item[(ii)] $\Soc(A,\sigma^q) \neq \{0\}$.
\end{itemize}
\end{definition}

\begin{proposition} \label{reduce-to-pp}
Let $\tau \in \Symo(A)$ and let $\sigma \in \Autmag(A,\tau)$ with $\Soc(A,\sigma)=\{0\}$. Then there is some power $\sigma^k$ of $\sigma$ such that $\Soc(A,\sigma^k)=\{0\}$ and $\sigma^k$ either has prime order or is a special permutation.
\end{proposition}
\begin{proof} 
Let $\sigma$ have order $m$ with prime factorisation $m=q_1^{e_1} \cdots q_r^{e_r}$. For $1 \leq i \leq r$, let $m_i \in \Z$ satisfy  $m_i \equiv 1 \pmod{q_i^{e_i}}$ and $m_i \equiv 0 \pmod{q_j^{e_j}}$ for $j \neq i$. Set $\sigma_i = \sigma^{m_i} \in \Autmag(A,\tau)$. Then $\sigma_i$ has order $q_e^{e_i}$ and $\sigma=\sigma_1 \cdots \sigma_r$. If $\Soc(A,\sigma_i) \neq \{0\}$ for all $i$ then, applying Proposition \ref{compose-soc} repeatedly, we have $\Soc(A,\sigma) \neq \{0\}$, contrary to our hypotheses. Thus, for some $i$, we have $\Soc(A,\sigma_i)=\{0\}$. To simplify notation, write $q=q_i$, $e=e_i$, so $\sigma_i$ has order $q^e$. If $\Soc(A,\sigma_i^{q^{e-1}})=\{0\}$ then take $k=m_i q^{e-1}$ so that $\sigma^k$ has prime order $q$. Otherwise let $h$ be minimal such that $\Soc(A,\sigma_i^{q^h})\neq \{0\}$. Then $1 \leq h <e$ and, taking $k=m_i q^{h-1}$, we have that $\sigma^k$ is a special permutation of order $q^{e-h+1}$. 
\end{proof}

\begin{proposition} \label{prime-special}
If Conjecture \ref{main-conj} does not hold for $A$, then there are $\tau$, $\sigma$ satisfying (\ref{st-conds}) such that each of $\tau$, $\sigma$ either has prime order or is a special permutation.
\end{proposition}
\begin{proof}
As Conjecture \ref{main-conj} does not hold for $A$, there are $\sigma$, $\tau$ satistifying (\ref{st-conds}).
By Proposition \ref{reduce-to-pp}, we may replace $\sigma$ by a suitable power $\sigma^k$ so that $\sigma$ either has prime order or is a special permutation. We still have $\sigma \in \Autmag(A,\tau)$, so, by Proposition \ref{recip}, also $\tau \in \Autmag(A,\sigma)$. Applying Proposition \ref{reduce-to-pp} again (with $\sigma$ and $\tau$ interchanged), we can replace $\tau$ by a suitable power $\tau^j$ so that also either $\tau$ has prime order or $\tau$ is a special permutation.
\end{proof}

\begin{corollary} \label{special_or_smallprime}
Let $A$ be cyclic of order $p^d \geq 5$, with $p \geq 3$. If Conjecture \ref{main-conj} does not hold for $A$ then there is a counterexample $\tau$ for $A$ which either has prime order $q \leq (p^d-1)/2$ or is a special permutation. 	
\end{corollary}
 \begin{proof}
 Let $\sigma$, $\tau$ be as in Proposition \ref{prime-special}. Then $\tau$ is as required unless $\tau$ has prime order $q \geq (p^d+1)/2$. In that case, $\sigma$ is either a special permutation or has prime order $r \leq (p^d-1)/2$ by Lemma \ref{long-cycle}, and $\tau \in \Autmag(A,\sigma)$ by Proposition \ref{recip}. Thus $\sigma$ is a counterexample with the required properties.	
 \end{proof}
 
 We now describe the special permutations in the case that $A$ has prime order.
 
 \begin{proposition} \label{special-p}
 Let $A$ be a cyclic group of order $p$ and let $\sigma \in \Symo(A)$ be a special permutation of order $q^e$ with $q$ prime and $e \geq 2$. Then $\sigma^q \in \Autgp(A)$, $q^e$ divides $p-1$, and $\sigma$ consists of $(p-1)/q^e$ cycles of length $q^e$.
 \end{proposition}
 \begin{proof}
 Since $\Soc(A,\sigma^q) \neq \{0\}$, we have $\Soc(A,\sigma^q)=A$, so $\sigma^q$ is an automorphism of $A$ of order $q^{e-1}>1$. Thus $\sigma^q$ has no fixed points except $0$ and consists of $(p-1)/q^{e-1}$ cycles of length $q^{e-1}$. It follows that $\sigma$ must consist of $(p-1)/q^e$ cycles of length $q^e$. In particular, $q^e$ divides $p-1$. 
 \end{proof}

\subsection{Ad hoc arguments for $|A|=17$ or $19$}

\begin{lemma} \label{ad hoc}
Let $A$ be the cyclic group of order $p=17$ or $19$. Suppose that $A$ does not satisfy Conjecture \ref{main-conj}. Then there is a counterexample $\tau$ for $A$ satisfying one of the following conditions:
\begin{enumerate}
	\item[(i)] $\tau$ is a special permutation;
	\item[(ii)] $\tau$ has order $2$ or $3$;
	\item[(iii)] $\tau$ is a single cycle of length $q=5$ or $7$;
	\item[(iv)] $\tau$ contains a cycle $\zeta$ of length $q=5$ or $7$ for which $\Supp(\zeta)$ and $\Supp(\tzeta)$ are disjoint, and $\tau = \zeta \tzeta^h$ with $1 \leq h \leq q-1$. 
\end{enumerate}
\end{lemma}
\begin{proof}
We treat the cases $p=17$ and $p=19$ simultaneously. Again let $\sigma$, $\tau$ satisfy (\ref{st-conds}). By Proposition \ref{prime-special} we may assume that each of $\sigma$ and $\tau$ either has prime order or is a special permutation, and we are free to swap $\sigma$ and $\tau$ by Proposition \ref{recip}. If either $\sigma$ or $\tau$ is a special permutation, or has order $2$ or $3$, we have a counterexample satisfying (i) or (ii). From now on we assume that $\tau$, $\sigma$ have prime orders $q$, $r \geq 5$ respectively.

If $q \geq 11$ then, by Lemma \ref{long-cycle}, $\sigma$ is supported on a set of cardinality at most $6$, so $r \leq 5$. If $r \neq 5$ then $\sigma$ satisfies (ii). If $r=5$ then $\sigma$ is a single cycle of length $5$, so satisfies (iii). Similarly, if $r \geq 11$ then $\tau$ satisfies (ii) or (iii). 

It remains to consider the cases $q$, $r \in \{5,7\}$. Moreover, we assume each of $\tau$, $\sigma$ contains at least two cycles, since otherwise $\tau$ or $\sigma$ satisfies (iii). We treat the cases $q=5$ and $q=7$ separately

 Suppose that $q=5$. Then  $\tau=\zeta_1 \zeta_2 \zeta_3$ or $\tau=\zeta_1 \zeta_2$ where the $\zeta_i$ are pairwise disjoint cycles of length $5$. 

First let $\tau = \zeta_1 \zeta_2 \zeta_3$. Since $\sigma$ commutes with $\tau$ and has order $5$ or $7$, it follows from Proposition \ref{perm-comm} that $\sigma= \zeta_1^a \zeta_2^b \zeta_3^c$ where $0 \leq a, b, c \leq q-1$. Relabelling the $\zeta_i$ and replacing $\sigma$ by a power, we assume without loss of generality that $a=1$. Then $\sigma= \zeta_1 \zeta_2^b \zeta_3^c$. As $\sigma$ also commutes with $\ttau = \tzeta_1 \tzeta_2 \tzeta_3$, we have $\zeta_1^d=\tzeta_i$ for some some $d \in \{1,2,3,4\}$ and $i \in \{1,2,3\}$. Since $\zeta_1$ has odd length, we cannot have $i=1$, so without loss of generality, $i=2$. Thus we have
$$ \tau = \zeta_1 \tzeta_1^d \zeta_3, \qquad \sigma = \zeta_1 \tzeta_1^{db} \zeta_3^c. $$
Now $\zeta_3$ is disjoint from $\zeta_1$ and $\zeta_2=\tzeta_1^d$, and hence $\tzeta_3$ is also disjoint from $\zeta_1$ and $\zeta_2$. As $p \leq 19$, it follows that $\zeta_3$ and $\tzeta_3$ cannot be disjoint. 
Since $\sigma$ commutes with $\ttau$, this means that $c=0$. Thus $\sigma$ is as in (iv). 

Now let $\tau=\zeta_1 \zeta_2$. By Proposition \ref{perm-comm}, $\sigma = \zeta_1^a \zeta_2^b \pi$ where $\pi$ has support of size at most $8$. If $a=b=0$ then $\sigma=\pi$, and, since $\sigma$ has order $5$ or $7$, it must consist of a single cycle, contrary to hypothesis. So, without loss of generality, we may assume that $a=1$. Then $r \neq 7$, so $r=5$ and either $\pi=\id$ or $\pi$ is a single cycle of length $5$ (disjoint from $\zeta_1$ and $\zeta_2$). However, if $\pi$ is a cycle of length $5$ and $b \neq 0$ then $\sigma$ consists of three cycles of length $5$, and, interchanging $\sigma$ and $\tau$, we are reduced to the previous case. We may therefore assume that $\sigma$ contains exactly two cycles, and either $\sigma = \zeta_1 \zeta_2^b$ with $b \neq 0$ or $\sigma = \zeta_1 \zeta'$ where $\zeta'$ is a further $5$-cycle disjoint from $\zeta_1$ and $\zeta_2$. Moreover $\sigma$ commutes with $\ttau = \tzeta_1 \tzeta_2$.

If $\sigma = \zeta_1 \zeta_2^b$ then $\zeta_2=\tzeta_1^d$ for some $d$, since the cycles $\zeta_1$, $\zeta_2$, $\tzeta_1$, $\tzeta_2$ cannot be pairwise disjoint. So $\tau$ is as in (iv). 

If $\sigma = \zeta_1 \zeta'$ then, for some $d$, we have $\zeta_1^d=\tzeta_2$ or $\zeta'=\tzeta_1^d$ or $\zeta'=\tzeta_2^d$, so that, respectively, $\tau=\zeta_1 \tzeta_1^d$ or $\sigma=\zeta_1 \tzeta_1^d$ or 
$\sigma=\zeta_1 \tzeta_2^d$. In the first two of these cases, either $\tau$ or $\sigma$ is as in (iv). We show that the remaining case $\tau=\zeta_1 \zeta_2$, $\sigma=\zeta_1 \tzeta_2^d$ cannot occur. Indeed, for $\tau$ and $\sigma$ to commute, $\zeta_1$, $\zeta_2$, $\tzeta_2$ must be pairwise disjoint, and if $\sigma$ and $\ttau$ commute then also $\tsigma$ and $\tau$ commute, so $\tzeta_1$, $\zeta_2$, $\zeta_1$ are pairwise disjoint. This would give $4$ pairwise disjoint $5$-cycles, which is impossible.

Finally, suppose that $q=7$. We may assume that $\tau=\zeta_1 \zeta_2$, the product of two cycles of length $7$. Then $\sigma = \zeta_1^a \zeta^b \pi$ where the support of $\pi$ has cardinality at most $4$. If $a=b=0$ then $r \leq 3$. So without loss of generality, $\sigma=\zeta_1 \zeta_2^b$. If $b=0$ then $\sigma$ satisfies (iii), so 
we can assume $1 \leq b \leq 6$. Since $\sigma$ commutes with $\ttau = \tzeta_1 \tzeta_2$ and there cannot be more than $2$ mutually disjoint cycles of length $7$, we have that $\tzeta_1$ is a power of $\zeta_2$ and $\tau$ satisfies (iv).
\end{proof}

\section{Computational Approaches}

In this section we describe our computational work, outlining two algorithms for testing Conjecture \ref{main-conj} and summarising the results obtained. We also describe some improvements to the second algorithm in the case of groups of prime order. These algorithms were implemented in Python and the approximate timings recorded were obtained on a laptop running at 2.80 GHz. We were able to verify Conjecture \ref{main-conj} for all abelian groups of order $\leq 14$ and also for the cyclic groups of order $17$ and $19$.

\subsection{Naive verification}

Let $A$ be a finite abelian group of order $N$. A very simple-minded way to test Conjecture \ref{main-conj} for $A$ is to generate a list of the $(N-1)!$ elements of $\Symo(A)$, and test all $(N-1)!^2$ pairs $(\tau,\sigma)$. If $\sigma \in \Autmag(A,\tau)$ and $\Soc(A,\tau)=\Soc(A,\sigma)=\{0\}$ then we have found a counterexample. If moreover $\sigma=\tau$ then we have found a quasi-linear cycle set with trivial socle, giving a negative answer to Question \ref{Rump-qn}.

There are several easy savings that can be made. Firstly, by Corollary \ref{conj-orbits} it is sufficient to consider one $\tau$ from each orbit under conjugation by $\Autgp(A)$. Secondly, by Proposition \ref{recip}, we may ignore the pairs $(\tau,\sigma)$ where $\sigma$ strictly precedes $\tau$ in our list. (We write this $\sigma \prec \tau$.) Thirdly, we omit those $\tau$ with  $\Soc(A,\tau) \neq \{0\}$, since these cannot give counterexamples to Conjecture \ref{main-conj}. While we could also omit those $\sigma$ with $\Soc(A,\sigma) \neq \{0\}$, this would result in our code producing no output unless a counterexample were found.  We prefer instead to print those pairs $(\tau,\sigma)$ where $\sigma \in \Autmag(A,\tau)$ and $\Soc(A,\tau)=\{0\}$. (We expect these to have $\Soc(A,\sigma) \neq \{0\}$.) This provides both a check that the code is working correctly and an indication during running of how far the calculation has progressed. 

The above considerations gives us Algorithm \ref{Alg1}.

\begin{algorithm}  \label{Alg1}
\caption{Naive test of Conjecture \ref{main-conj}}
\begin{algorithmic}
\State Given a finite abelian group $A$ of order $N$:
\State Create list of elements of $\Symo(A)$.
\State Create group multiplication table and inverse table for $A$.
\State Create table of group automorphisms of $A$ and their inverses.
\For{$\tau \in \Symo(A)$}
    \If{$\Soc(A,\tau)=\{0\}$ and no $\alpha \in \Autgp(A)$ has $\alpha \tau \alpha^{-1} \prec \tau$}
        \For{$\sigma \in \Symo(A)$ with $\tau \preceq \sigma$}
            \If{$\sigma \in \Autmag(A,\tau)$}
                  \State Print $(\tau,\sigma)$ and check if $\Soc(A,\sigma)=\{0\}$.
            \EndIf
          \EndFor
    \EndIf
\EndFor
\end{algorithmic}
\end{algorithm}

We give a little more detail on how this was implemented. 
We fix an enumeration $a_0=0, \ldots, a_{N-1}$ of the elements of $A$ and provide a Python function to implement the group operation. (This is just addition mod $N$ if $A$ is cyclic). We then generate a list of the elements $\tau$ of $\Symo(A)$ in lexicographical order by a backtracking algorithm, where $\tau$ is represented as a list of indices $[v_0, \ldots, v_{N-1}]$ with $a_{v_i}=\tau(a_i)$, so in particular $v_0=0$. This allows the conditions 
$\alpha \tau \alpha^{-1} \prec \tau$ and $\tau \preceq \sigma$ to be tested without searching through the list. Permutations are converted to cycle notation when they are printed.

Table \ref{Alg1-table} shows the results of Algorithm \ref{Alg1} for all abelian groups $A$ with $|A| \leq 10$. We omit the cases with $|A| \leq 3$ and $A=C_2 \times C_2$ since in these cases $\Symo(A)=\Autgp(A)$ and Conjecture \ref{main-conj} is vacuously satisfied. No counterexamples to Conjecture \ref{main-conj} were found.

\begin{table}[ht]
	\centerline{
		\begin{tabular}{|c|c|c|c|} \hline
		Group & $\#\ \tau$ tested & $\# $ magma automorphisms & Time   \\ \hline
		$C_4$ & $2$ & $0$ & $0.1$ seconds \\
		$C_5$ & $6$ & $2$ & $0.15$ seconds \\
		$C_6$ & $56$ & $9$ & $0.3$ seconds \\
		$C_7$ & $130$ & $22$ & $0.7$ seconds \\
		$C_8$ & $1\,276$ & $55$ & $6$ seconds \\
    	$C_4 \times C_2$ & $596$ & $24$ & $3$ seconds \\
		$C_2 \times C_2 \times C_2$ & $28$ & $23$ & $1$ second \\
		$C_9$ & $6\,784$ & $215$ & $3$ minutes \\
		$C_3 \times C_3$ & $866$ & $129$ & $28$ seconds \\
		$C_{10}$ & $90\,720$ & $358$ & $6$ hours \\
			\hline
		\end{tabular}
	}  
	\vskip4mm
\caption{Results from Algorithm \ref{Alg1}} \label{Alg1-table}
\end{table}	

The second column  of Table \ref{Alg1-table} shows the number of elements $\tau \in \Symo(A)$ tested by the algorithm. For example, when $A=C_5$ there are $20$ elements of $\Symo(C_5)$ which are not group automorphisms. These fall into $6$ orbits under conjugation by $\Autgp(C_5)$, represented by $(34)$, $(23)$, $(234)$, $(243)$, $(12)(34)$, $(1234)$. These $6$ permutations are tested to see if any $\sigma \succeq \tau$ is a magma automorphism of $(A,\tau)$. This occurs only for $\tau=(23)$ and $\tau=(12)(34)$, both with $\sigma=(14)(23)$. (Note that, as expected, $\sigma \in \Autgp(A)$.) Thus two magma automorphisms are found, as shown in the third column. In general, there may be many magma automorphisms $\sigma$ for the same $\tau$.

\subsection{Using Magma generators}

Algorithm \ref{Alg1} is slow because it finds magma automorphisms for $(A,\tau)$ by searching through all $\sigma \succeq \tau$. It also requires a stored list of all $\sigma \in \Symo(A)$. A more efficient method of searching for magma automorphisms, which does not require such a list,  is to use a (small) set of generators of the magma $(A,\tau)$. To search for $\sigma\in \Autmag(A,\tau)$ in this way, we need to know not only the magma generators, but also how each element can be obtained from these generators. We record this information in what we call a magma generation table. Table \ref{mag_gen_table} gives an example of such a table for $A=C_{10}$ with $\tau=(68)$. The multiplication table for this magma is shown in Table \ref{mag_gen_ex}. 

\begin{table}[ht]
	\centerline{
		\begin{tabular}{|c|c|c|} \hline
			product & op1 & op2  \\ \hline
			$0$ &	$0$ 	& 	$0$		 \\ 
			$2$ &	$-1$	&	$-1$	 \\
			$4$ &	$2$		&	$2$		 \\
			$6$ &	$2$		&	$4$		 \\
			$8$ &	$4$		&	$2$		 \\
			$1$ &   $-1$	&	$-1$		 \\
			$3$ &	$2$	&	$1$	 \\
			$9$ &	$4$		&	$1$		 \\
			$5$ &	$2$		&	$3$		 \\
			$7$ &	$4$		&	$9$		 \\
			\hline
		\end{tabular}
	}  
	\vskip4mm
	\caption{Magma generation table for $A=C_{10}$, $\tau=(68)$} \label{mag_gen_table}
\end{table}	

\begin{table}[ht]
	\centerline{
		\begin{tabular}{|c|cccccccccc|} \hline
			& $0$ & $1$ & $2$ & $3$ & $4$ & $5$ & $6$ & $7$ & $8$ & $9$  \\ \hline
			$0$ &	$0$ & $1$ & $2$ & $3$ & $4$ & $5$ & $8$ & $7$ & $6$ & $9$ \\
			$1$ &  $0$ & $1$ & $2$ & $3$ & $4$ & $5$ & $6$ & $9$ & $8$ & $7$    \\
			$2$ &	$0$ & $3$ & $4$ & $5$ & $6$ & $7$ & $8$ & $9$ & $2$ & $1$  \\
			$3$ &	$0$ & $9$ & $2$ & $3$ & $4$ & $5$ & $6$ & $7$ & $8$ & $1$  \\
			$4$ &	$0$ & $9$ & $8$ & $1$ & $2$ & $3$ & $4$ & $5$ & $6$ & $7$  \\
			$5$ &	$0$ & $3$ & $2$ & $1$ & $4$ & $5$ & $6$ & $7$ & $8$ & $9$  \\
			$6$ &	$0$ & $1$ & $4$ & $3$ & $2$ & $5$ & $6$ & $7$ & $8$ & $9$  \\
			$7$ &	$0$ & $1$ & $2$ & $5$ & $4$ & $3$ & $6$ & $7$ & $8$ & $9$  \\
			$8$ &	$0$ & $1$ & $2$ & $3$ & $6$ & $5$ & $4$ & $7$ & $8$ & $9$   \\
			$9$ &	$0$ & $1$ & $2$ & $3$ & $4$ & $7$ & $6$ & $5$ & $8$ & $9$  \\
			
			\hline
		\end{tabular}
	}  
	\vskip4mm
	\caption{Magma operation for $A=C_{10}$, $\tau =(68)$} \label{mag_gen_ex}
\end{table}	

We explain the format of a magma generation table. The first column lists the elements of $A$ (in some order) starting with $0$. The value $-1$ in the last two columns indicates the introduction of a new generator, so the 2nd and 6th rows of Table \ref{mag_gen_table} show that our generating set is $\{2,1\}$. Every other row (except the first) shows how the element in the first column is obtained from previous elements. Thus, in our example, the submagma generated by $2$ contains $4=2 \cdot 2$, $6=2 \cdot 4$ and $8=4 \cdot 2$. In fact, this submagma is precisely $\{0,2,4,6,8\}$ so we require a further generator. Choosing $1$, we then obtain the remaining elements of $C_{10}$. 

To construct a magma generation table for a given $\tau$, we first check whether the magma $(A,\tau)$ is cyclic. By Lemma \ref{submags}(i) it suffices to take one element $a$ from each orbit of $A$ under the subgroup $\langle \tau, \ttau \rangle$ of $\Symo(A)$ and test whether $a$ generates the full magma $(A,\tau)$. If so, then we can build a magma generation table using the single generator $a$. Otherwise, we take the generator $a$ which gave the largest cyclic submagma, and add further generators until we do obtain a generating set. Thus, in Table \ref{mag_gen_table}, we started with the generator $2$ since this generates the cyclic submagma $\{0,2,4,6,8\}$ of size $5$, while $1$ generates the smaller cyclic submagma $\{0,1\}$, and no element generates a cyclic submagma of size larger than $5$. This process for choosing the generators is not guaranteed to find a generating set of minimal size (except when a single generator suffices), but we suspect that it will do so in nearly all cases.

Having found the magma generation table, we search for magma automorphisms $\sigma$ by considering all ways to assign values of $\sigma$ on the generators. We illustrate this with Table \ref{mag_gen_table}. 

We first try
 $\sigma(2)=1$. The 3rd row of Table \ref{mag_gen_table}, together with Table \ref{mag_gen_ex}, then gives 
 $$ \sigma(4) = \sigma(2) \cdot \sigma(2) = 1 \cdot 1 =  1 = \sigma(2). $$
 This shows no injective function $\sigma$ with $\sigma(2)=1$ is compatible with Table \ref{mag_gen_table}. 
 
 We next try $\sigma(2)=2$. We then get
 $$ \sigma(4) = \sigma(2) \cdot \sigma(2) = 2 \cdot 2 = 4 $$
 and similarly $\sigma(6)=6$, $\sigma(8)=8$. Thus $\sigma$ fixes pointwise the submagma $\{0,2,4,6,8\}$ generated by $2$. We then consider the possibilities for $\sigma(1)$; these are the elements $1$, $3$, $5$, $7$, $9$ not already assigned as values of $\sigma$. If $\sigma(1)=1$, we have
 $$ \sigma(3) = \sigma(2) \cdot \sigma(1) = 2 \cdot 1 = 3, $$
 and similarly $\sigma(9)=9$, $\sigma(5)=5$, $\sigma(7) =7$. Thus $\sigma$ is the identity automorphism (which we do not record). If $\sigma(1)=3$, we obtain
 $$ \sigma(3) = \sigma(2) \cdot \sigma(1) = 2 \cdot 3 = 5, $$
 $$ \sigma(9) = \sigma(4) \cdot \sigma(1) = 4 \cdot 3 =1, $$
 $$ \sigma(5) = \sigma(2) \cdot \sigma(3) = 2 \cdot 5 =7, $$
 $$ \sigma(7) = \sigma(4) \cdot \sigma(9) = 4 \cdot 1 =9. $$
 Thus the permutation $\sigma=(13579)$ is compatible with Table \ref{mag_gen_table}. We check that $\sigma$ is indeed a magma autumorphism. Assigning $\sigma(1)=5$, $7$, $9$ gives powers of this permutation.
 
 We next try $\sigma(2)=3$. This immediately gives $\sigma(4)=3 \cdot 3 = 3=\sigma(2)$, so we do not obtain a bijection.
 
 Taking $\sigma(2)=4$, we get $\sigma(4)=2$, $\sigma(6)=8$, $\sigma(8)=6$. If $\sigma(1)=1$, we obtain $\sigma=(24)(68)(39)(57)$ which is not a magma automorphism since for example $\sigma(3) \cdot \sigma(1) = 9 \cdot 1 =1$ but $\sigma(3 \cdot 1) = \sigma(9)=3$. If $\sigma(1)=3$, we obtain $\sigma=(24)(68)(13)(59)$, which again is not a magma automorphism. Carrying on in this way, we find that $\Autmag(C_{10},\tau)$ is cyclic of order $5$, generated $(13579)$. 
 
The method just outlined for finding magma automorphisms can also be used to find the automorphisms of the group $A$ itself via a magma generation table for the group.

These considerations give us Algorithm \ref{Alg2}.

\begin{algorithm}  \label{Alg2}
	\caption{Test of Conjecture \ref{main-conj} using magma generators}
	\begin{algorithmic}
		\State Given a finite abelian group A:
		\State Create group multiplication table and inverse table for $A$.
		\State Create table of group automorphisms of $A$ and their inverses.
		\For{$\tau \in \Symo(A)$}
		\If{$\Soc(A,\tau)=\{0\}$ and no $\alpha \in \Autgp(A)$ has $\alpha \tau \alpha^{-1} \prec \tau$}
		\State Create magma generation table for $(A,\tau)$
		\State Create list of elements of $\Autmag(A,\tau)$
		\For{$\sigma \in \Autmag(A,\tau) \bs \{\id\}$}
		\State Print $(\tau,\sigma)$ and check if $\Soc(A,\sigma)=\{0\}$.
		\EndFor
		\EndIf
		\EndFor
	\end{algorithmic}
\end{algorithm}

Algorithm \ref{Alg2} will test some pairs $(\tau, \sigma)$ with $\sigma \prec \tau$, so does not give the same numerical results as Algorithm \ref{Alg1}. Most pairs tested will however satisfy
$\tau \preceq \sigma$ since we only test the first $\tau$ (with respect to the ordering $\prec$) in each orbit under conjugation by $\Autgp(A)$. For comparability with Algorithm \ref{Alg1}, we include a count those pairs for which $\tau \prec \sigma$ does hold. For $|A| \leq 8$, Algorithm \ref{Alg2} recovers the results of Algorithm \ref{Alg1} in at most $2$ seconds. We show the results and timings for $9\leq |A| \leq 14$ in Table \ref{Alg2-table}.

\begin{table}[ht]
	\centerline{
		\begin{tabular}{|c|c|c|c|c|c|} \hline
			Group & $\#\ \tau$ tested & $\#$ with & $\# $ magma & $\# $ magma  & Time   \\ 
			   &  &  noncyclic & automorphisms & automorphisms  & \\
			    & &  magma & & with $\tau \preceq \sigma$ & \\ \hline		
			$C_9$ & $6\,784$ & $28$ & $260$ & $215$ & $7$ seconds \\
			$C_3 \times C_3$ & $866$ & $21$ & $190$ & $5$ & $5$ seconds \\
			$C_{10}$ & $90\,720$ & $730$ & $408$ & $358$ & $18$ seconds \\
			$C_{11}$ & $363\,278$ & $0$ & $846$ & $820$ & $1$ minute \\
			$C_{12}$ & $9\,981\,000$ & $25\,069$ & $9\,608$ & $7\,669$ & $21$ minutes \\
			$C_6 \times C_2$ & $3\,326\,020$ & $23\,273$ & $8\,464$ & $4\,526$ & $11$ minutes \\
			$C_{13}$ & $39\,921\,032$ & $0$ & $9\,064$ & $8\,773$ & $1.7$ hours \\
			$C_{14}$ & $1\,037\,837\,374$ & $605\,163$ & $22\,754$ & $21\,706$ & $36$ hours \\
			\hline
		\end{tabular}
	}  
\vskip4mm
\caption{Results from Algorithm \ref{Alg2}} \label{Alg2-table}
\end{table}

\subsection{Groups of prime order}

In this final subsection, we describe some modifications to Algorithm \ref{Alg2} in the case that $A=C_p$ with $p$ prime. By Proposition \ref{prime-special} and Corollary \ref{special-p}, it suffices to consider only those $\tau$ which have prime order $q \leq (p-1)/2$ or are special permutations of prime-power order $q^e$ dividing $p-1$. By verifying that no such $\tau$ is a counterexample, we are then able to verify Conjecture \ref{main-conj} for $A=C_{11}$ and $A=C_{13}$ more quickly than with Algorithm \ref{Alg2}. By incorporating Lemma \ref{ad hoc}, we are also able to handle $C_{17}$ and $C_{19}$.

Recall from Corollaries \ref{conj-orbits} and \ref{cyclic-powers} that if $\tau$ is a counterexample for $A$ then so are $\tau^k$ and $\alpha \tau \alpha^{-1}$ for any $k$ coprime to the order of $\tau$ and for any $\alpha \in \Autgp(A)$. In Algorithm \ref{Alg2} we used a backtracking process to generate all $\tau \in \Symo(A)$, but then ignored those with  $\alpha \tau \alpha^{-1} \prec \tau$ for some group automorphism $\alpha$. We now adapt the backtracking process to generate only those permutations with cycle decompositions of the required form which satisfy certain additional conditions. These additional conditions, which are different in the two cases of permutations of prime order and special permutations, guarantee that, for any $\tau$ of the appropriate type, at least one of the permutations generated can be obtained from $\tau$ by a combination of taking powers and conjugating by group automorphisms. Thus the set of permutations $\tau$ generated will be sufficient to show that there are no counterexamples of the type considered.

\subsubsection{Special Permutations}

We deal first with special permutations, since there are few of these and they have a known cycle shape. Let $q^e$ be a prime power dividing $p-1$, with $e \geq 2$. We fix a group automorphism $\alpha$ of $A$ of order $q^{e-1}$. Let $\tau$ be a special permutation of order $q^e$. Since $\Autgp(A)$ is cyclic of order $p-1$, it follows from Proposition \ref{special-p} that $\tau^q=\alpha^k$ for some $k$ with $\gcd(k,q)=1$. Replacing $\tau$ by $\tau^m$ where $km \equiv 1 \pmod{q^e}$, we may assume that $\tau^q=\alpha$.  

The permutation $\tau$ can be written as a product of $(p-1)/q^e$ cycles of length $q^e$:
\begin{equation} \label{spec-cycles}
	\tau = (u_1, \ldots, u_{q^e}) (u_{q^e+1},  \ldots, u_{2q^e}) \cdots
	(u_{p-1-q^e+1}, \ldots, u_{p-1}),
\end{equation} 
where $u_1, \ldots, u_{p-1}$ are the elements of $\{1, \ldots, p-1\}$ in some order.
We write this decomposition in the standard way, so that the smallest element of each cycle occurs at the start of the cycle and the cycles are ordered with their first elements occurring in increasing order. Thus we have
\begin{equation} \label{spec-cycles-in-ord}
	u_{i q^e+j} > u_{i q^e +1} \mbox{ for } 0 \leq i < (p-1)/q^e \mbox{ and } 2 \leq j \leq q^e.
\end{equation}
\begin{equation} \label{spec-first-in-ord}
	 u_1=1 < u_{q^e+1} < \cdots < u_{p-1-q^e+1}, 
\end{equation}
The condition $\tau^q=\alpha$ means that each cycle is determined by its first $q$ elements: 
\begin{equation}  \label{q-pow-is-tau}
  u_{i q^e + j} = \alpha(u_{iq^e +j-q})  \mbox{ for } 0 \leq i < (p-1)/q^e \mbox{ and } q+1 \leq j \leq q^e.
\end{equation}
Finally, we are free to replace $\tau$ by $\tau^{1+q^{e-1}h}$ since this preserves the condition $\tau^q=\alpha$. We choose $h$ to minimise $u_2$, so 
\begin{equation} \label{2nd-min}
	u_{2+q^{e-1} h} > u_2 \mbox{ for } 1 \leq h \leq q-1.
\end{equation}
The conditions (\ref{spec-cycles-in-ord}) and (\ref{spec-first-in-ord}) ensure we generate each permutation $\tau$ of the form (\ref{spec-cycles}) only once, whereas 
(\ref{q-pow-is-tau}) and (\ref{2nd-min}) are the additional conditions reducing the number of permutations generated.

We modify the backtracking algorithm used previously to generate all sequences $[0, u_1,\ldots, u_{p-1}]$ satisfying (\ref{spec-cycles-in-ord})--(\ref{2nd-min}). We convert each of these sequences, representing $\tau$ as a cycle decomposition (\ref{spec-cycles}), to the sequence
$[0, v_1,\ldots, v_{p-1}]$ where $\tau(a)=v_a$, as used previously. We then find $\Autmag(A,\tau)$ as before. This gives us the variant of Algorithm \ref{Alg2} which we list as Algorithm 6.3.

\begin{algorithm} \label{Alg3}
	\caption{Test of Conjecture \ref{main-conj} for special permutations}
	\begin{algorithmic}  
		\State Given $A=C_p$ for $p$ prime:
		\State Create group multiplication table and inverse table for $A$.
		\For {each prime power $q^e$ (with $e \geq 2$) dividing $p-1$}
		\For{$\tau \in \Symo(A)$ with cycle decomposition as in (\ref{spec-cycles})--(\ref{2nd-min})}
		\If{$\Soc(A,\tau)=\{0\}$}
		\State Create list of elements of $\Autmag(A,\tau)$
		\For{$\sigma \in \Autmag(A,\tau) \bs \{\id\}$}
		\State Print $(\tau,\sigma)$ and check if $\Soc(A,\sigma)=\{0\}$.
		\EndFor
		\EndIf
		\EndFor
		\EndFor
	\end{algorithmic}
\end{algorithm}

Table \ref{spec-table} shows the timings to test the special permutations for primes $p$ with $13 \leq p \leq 29$ for which $p-1$ is not squarefree. Although the full verification of Conjecture \ref{main-conj} for primes $p>19$ is not viable with our methods, we include the special permutations for $p=29$ (the next case where $p-1$ is not squarefree) since these can be handled in a reasonable time. The condition $\Soc(A,\tau)=\{0\}$ in Algorithm 6.3 ensures that special permutations which are already group automorphisms are excluded. No counterexamples were found among the remaining special permutations.

\begin{table}[ht] \label{spec-table}
	\centerline{
		\begin{tabular}{|c|c|c|c|c|} \hline
			$p$ & Order(s) of special  & $\#\ \tau$ tested & \# magma & Time\\ 
				&	permutations &    & automorphisms & \\ \hline
			$13$ & $4$  & $60$ & $64$ & $0.3$ seconds \\      
			$17$ & $4$, $8$, $16$ & $868$ & $652$ & $0.5$ seconds \\
			$19$ & $9$ & $1\,080$ & $1\,158$ & $0.7$ seconds \\
		    $29$ & $4$ & $8\,648\,640$ & $2\,483\,754$     & $27$ minutes \\
			\hline
		\end{tabular}
	}  
\vskip4mm
\caption{Results for Special Permutations} \label{spec-table}
\end{table}

\subsubsection{Permutations of prime order}
We next describe a modification to Algorithm \ref{Alg2} to handle permutations of arbitrary prime order $q<p$. Such a permutation may contain up to $\lfloor (p-1)/q \rfloor$ cycles of length $q$. 

We fix a primitive root $g$ modulo $p$ and generate sequences of length $p$ of the shape 
\begin{equation}  \label{pr-seq}
 	[0,u_1, u_2, \ldots, u_{kq}, 0 \ldots, 0], 
 \end{equation}
 where $1 \leq k \leq \lfloor (p-1)/q \rfloor$ and the final block of $0$'s has length $p-1-kq \geq 0$. Here $u_1, \ldots, u_{kq}$ are distinct elements of $\{1, 2, \ldots, p-1\}$. However, we now interpret this sequence as the permutation 
 \begin{equation} \label{cyc-tau-perm}
 	  \tau = (g^{u_1}, \ldots, g^{u_q}) \cdots (g^{u_{kq-q+1}} \ldots , g^{u_{kq}}) \in \Symo(A) 
 \end{equation}
 consisting of $k$ cycles of length $q$. Thus
 \begin{equation} \label{lambda-perm} 
  \lambda = (u_1, \ldots, u_q) \cdots (u_{kq-q+1}, \ldots , u_{kq})
 \end{equation}
 is the permutation of exponents on $g$ corresponding to $\tau$. We again assume this is written in the standard way:
 \begin{equation} \label{q-first-in-ord}
 	u_1=1 < u_{q+1} < \cdots < u_{kq-q+1}, 
 \end{equation}
 \begin{equation} \label{q-cycles-in-ord}
 	u_{i q+j} > u_{i q +1} \mbox{ for } 0 \leq i < k \mbox{ and } 2 \leq j \leq q.
 \end{equation}
We choose to work with $u_j$ in the range $1 \leq u_j \leq p-1$, rather than $0 \leq u_j \leq p-2$, because of the special role of $0$ as the first element of the sequence (\ref{pr-seq}). It is therefore convenient to introduce the notation $a \oplus b$, respectively $a \ominus b$, for the least strictly positive residue mod $p-1$ of $a+b$, respectively $a-b$. 
 
 Any automorphism of $A=C_p$ can be written as $a \mapsto ag^c$ in $\Z/p\Z$ for some fixed $c$. Conjugating $\tau$ by this automorphisms transforms the sequence (\ref{pr-seq}) to $[0,u_1\oplus c, u_2\oplus c, \ldots, u_{kq} \oplus c, 0 \ldots, 0]$. The advantage of working with the exponents $u_i$ is that the set $\{  \lambda(u_i) \ominus u_i : 1 \leq i \leq kq \}$ of jumps of $\lambda$ is preserved by conjugation by group automorphisms, and a suitable conjugation moves the minimum jump so that it occurs when $i=1$. 
 We may therefore impose the condition
 \begin{equation} \label{q-cycles-min}
 	u_{hq+j+1} \ominus u_{hq+j} \geq u_2-1 \mbox{ for } 0 \leq h < k \mbox{ and } 1 \leq j < q. 
 \end{equation}
Moreover, if for some $i \leq q$ we have $u_i<u_2$ (which can only occur if $u_{i-1}>u_i$) then we may replace $\lambda$ by a suitable power to move $u_i$ into the second position, thereby reducing the minimum jump of $\lambda$. Thus we may assume that 
\begin{equation} \label{u2-sec-min}
	u_2 = \min_{2 \leq i \leq q} u_i.
\end{equation} 
Our additional conditions to reduce the number of permutations $\tau$ of order $q$ tested are then (\ref{q-cycles-min}) and (\ref{u2-sec-min}).

We then have Algorithm 6.4.

\begin{algorithm}  \label{Alg4}
	\caption{Test of Conjecture \ref{main-conj} for $A=C_p$, $\tau$ of order $q$}
	\begin{algorithmic} 
		\State Given $A=C_p$ for $p$ prime and a prime $q<p$:
		\State Find a primitive root mod $p$ and create a table of its powers.
		\State Create group multiplication table and inverse table for $A$.
		\For {each permutation $\lambda$ as in (\ref{lambda-perm})--(\ref{u2-sec-min})}
		\State Construct $\tau$ corresponding to $\lambda$ by (\ref{cyc-tau-perm})
		\If{$\Soc(A,\tau)=\{0\}$}
		\State Create list of elements of $\Autmag(A,\tau)$
		\For{$\sigma \in \Autmag(A,\tau) \bs \{\id\}$}
		\State Print $(\tau,\sigma)$ and check if $\Soc(A,\sigma)=\{0\}$.
		\EndFor
		\EndIf
		\EndFor
	\end{algorithmic}
\end{algorithm}

 In Table \ref{gen-q-table} we show the results from Algorithm 6.4 for $p=11$, $13$, $17$, $19$. 
 For $p=11$ and $p=13$ we include all primes $q<p$, although by Corollary \ref{special_or_smallprime} only the primes $q \leq 5$ are necessary. For $p=17$ and $p=19$, the primes $q=2$, $3$, $5$, $7$ all need to be considered, but we only include $q=2$, $3$ here and handle $q=5$, $7$ in \S\ref{ad-hoc-sec} below. For $A=C_{11}$, verifying Conjecture \ref{main-conj} takes about $9$ seconds using Algorithm 6.4, compared to one minute using Algorithm \ref{Alg2}. Verifying Conjecture \ref{main-conj} for $A=C_{13}$ using Algorithms 6.3 and 6.4 takes less than $2$ minutes, compared to $1.7$ hours with Algorithm \ref{Alg2}. 
 
\begin{table}[ht]
	\centerline{
		\begin{tabular}{|c|c|c|c|c|c|} \hline
			$p$ & $q$  & $\#\ \tau$ tested & $\# $ magma  & Time   \\ 
			&  & & automorphisms  & \\ \hline
			& $2$ &  $1\,541$ &  $87$ & $3$ seconds    \\
			$11$ & $3$ & $4\,729$  &  $11$ & $1$ second   \\
			& $5$ & $11\,177$  & $83$ & $4$ seconds     \\
			& $7$ &  $8\,071$ & $0$  &  $1$ second \\ \hline
			& $2$ & $18\,405$  &  $350$ &  $14$ seconds  \\
			$13$ & $3$ & $100\,014$   & $275$  &  $21$ seconds   \\
			& $5$ & $530\,295$  & $300$  & $1$ minute   \\
			& $7$ & $40\,946$  & $0$  & $4$ seconds    \\
			& $11$ & $3\,762\,296$  & $0$  &  $7$ minutes   \\ \hline			
			$17$ & $2$ &  $4\,451\,105$  & $5\,091$ & $15$ minutes \\
			& $3$ &	$117\,715\,139$ & $2\,151$  & $5.2$ hours   \\ \hline
			$19$ & $2$ &  $85\,029\,389$  & $23\,197$  & $4.5$ hours   \\
			& $3$ & $4\,673\,448\,594$ & $28\,232$  & $11$ days   \\ \hline
		\end{tabular}
	}  
\vskip4mm
\caption{Results for cycles of prime order} \label{gen-q-table}
\end{table}

\subsubsection{Further modifications for $p=17$ and $p=19$}  \label{ad-hoc-sec}
To handle the cases $p=17$, $19$ with $q=5$, $7$, we further reduce the number of permutations to be tested by adapting Algorithm 6.4 in line with Lemma \ref{ad hoc}(iii), (iv). For Lemma \ref{ad hoc}(iii), 
we only consider permutations (\ref{lambda-perm}) with $k=1$, so that $\tau=\zeta$ is a single $q$-cycle. For Lemma \ref{ad hoc}(iv), we generate cycles $\zeta$ with the additional property that $\zeta$ and $\tzeta$ are disjoint, and then test all permutations $\tau=\zeta \tzeta^h$ for $1 \leq h \leq q-1$. The results are shown in Table \ref{ad hoc table}, and these complete the verification of Conjecture \ref{main-conj} for $A=C_{17}$ and $A=C_{19}$.

\begin{table}[ht]
	\centerline{
		\begin{tabular}{|c|c|c|c|c|c|} \hline
			$p$ & $q$  & Form of $\tau$ & $\#\ \tau$ tested & $\# $ magma  & Time   \\ 
			&  & & & automorphisms  & \\ \hline
		  $17$	& $5$ &  $\zeta$ & $5\,068$  & $0$  & $1$ second    \\
			 & & $\zeta \tzeta^h$ & $8\,436$ & $2\,109$  &  $2$ seconds   \\ \hline
			$17$ & $7$  & $\zeta$ & $403\,238$  &  $0$ &  $1$ minute  \\ 	
			 &   & $\zeta \tzeta^h$ & $223\,110$  &  $37\,185$ &  $38$ seconds  \\ \hline	
		 $19$	& $5$ & $\zeta$ & $8\,673$  & $0$ & $2$ seconds    \\
			 &  & $\zeta \tzeta^h$ & $16\,524$ & $4\,131$ & $4$ seconds  \\ \hline
			$19$ & $7$ & $\zeta$ & $967\,138$  & $0$  & $3$ minutes     \\
				&   & $\zeta \tzeta^h$ & $858\,576$ & $143\,096$  & $3$ minutes   \\ \hline		
		\end{tabular}
	}  
\vskip4mm
\caption{Results for $p=17$, $19$ and $q=5$, $7$} \label{ad hoc table}
\end{table}
\pagebreak

\bibliography{MagmaAutomsBib}

\end{document}